\documentclass[11pt]{article}

\title{Exploration of a Cosine Expansion Lattice Scheme}

\usepackage{fullpage}
\usepackage{amsfonts,amssymb,amsmath,amsthm}
\usepackage{graphicx}
\usepackage{caption}
\usepackage{color}
\usepackage{hyperref}
\usepackage{diagbox}
\usepackage{authblk}

\usepackage[utf8]{inputenc}
\usepackage[scaled=0.85]{beramono}

\allowdisplaybreaks[1]
\numberwithin{equation}{section}

\newtheorem{theorem}{Theorem}[section]

\newtheorem{proposition}[theorem]{Proposition}
\theoremstyle{definition}\newtheorem{definition}[theorem]{Definition}
\theoremstyle{definition}
\theoremstyle{remark}\newtheorem{remark}[theorem]{Remark}

\author[1]{Ki Wai Chau}
\author[1,2]{Cornelis W. Oosterlee}
\affil[1]{Delft Institute of Applied Mathematics, Delft University of Technology}
\affil[2]{Research Group of Scientific Computing, Centrum Wiskunde \& Informatica}

\begin{document}
\maketitle

\begin{abstract}
	In this article, we combine a lattice sequence from Quasi-Monte Carlo rules with the philosophy of the Fourier-cosine method to design an approximation scheme for expectation computation.
	We study the error of this scheme and compare this scheme with our previous work on wavelets.
	Also, some numerical experiments are performed.
\end{abstract}

\section{Introduction}

In this work, we derive a numerical integration formula for a function $\mathfrak{f}: \mathbb{R}^s \rightarrow \mathbb{R}$ with respect to a probability measure $\mu$. 
Namely, we aim to approximate the following quantity:
\begin{equation}
	\mathbb{E}[\mathfrak{f}({\bf Y})] = \int_{\mathbb{R}^s} \mathfrak{f}({\bf y})\mu(d {\bf y}). \label{equation_expectation}
\end{equation}
The inspiration of this work comes from two promising methods in numerical integration, the COS method and lattice rules. 

The COS method is a numerical integration method developed in the context of option pricing, see \cite{fang_2009_novel}. 
Within this framework, the fair value of a financial option can be expressed as the expectation of the discounted payoff function.
Using the fact that the payoff function can be expressed as a cosine series, an approximation of the integration in terms of the Fourier transform of the risk neutral measure and the cosine coefficient of the payoff function were presented.
This method is efficient in terms of calculation with only simple addition operations.
It also makes use of the characteristic function of a probability measure, which is available more often than the density function. 
Ever since, there has been numerous extensions of the COS method, including pricing Bermudan options \cite{fang_2009_pricing}, finding ruin probability \cite{chau_2015_ruin}, solving backward stochastic differential equations \cite{ruijter_2015_fourier} and computation of valuation adjustment \cite{borovykh_2018_efficient}.

There are also several applications for which it is necessary to extend the COS method beyond the one-dimensional situation, for example, in \cite{ruijter_2012_two}.
However, the tensor extension used in previous studies suffers from the curse of dimensionality, with the number of summation terms increasing exponentially when the number of dimensions increases. 
Therefore, further input is required for such extensions to be feasible in practice. 

Thus, we took the insight from Quasi-Monte Carlo (QMC) rules, which approximate integrals of the form $$\int_{[0,1]^s}\mathfrak{f}({\bf y})d{\bf y}$$ with equal weight quadrature rules $$\frac{1}{N}\sum^{N-1}_{n=0}\mathfrak{f}({\bf y}_n).$$
Readers are referred to \cite{dick_2010_digital}, \cite{niedereriter_1992_random} and references therein for further details. 

In particular, we are interested in the rank 1 lattice rule construction method for quadrature points, where for a given positive number $N$ under dimension $s$, one chooses a vector ${\bf g} \in \{1, \ldots, N-1\}^s$ called the generating vector and generates a points set:  $$\left(\left\{\frac{n{\bf g}}{N}\right\}\right)_{0 \leq n < N-1},$$ where the notation $\{\}$ denotes the fractional part of the real number in each dimension.
The quadrature points are the result of applying some fixed function on this points set. 

There have been numerous research papers supporting the implementation of lattice rule QMC, in identifying a suitable generating vector \cite{dick_2004_on, kuo_2002_component, sloan_2002_component,  sloan_2002_on}, in efficient algorithms to generate such vectors \cite{nuyens_2006_fast, nuyens_2006_fast_non_prime}, and in extensible lattice rules \cite{cools_2006_constructing, dick_2008_construction, hickernell_2012_weighted, hickernell_2003_existence}.
In \cite{dick_2014_lattice}, the authors derived an error bound for QMC with tent transformed lattice rules for functions within the half-period cosine space. 
Noting the connection between the half-period cosine space and the tensor extension of the cosine series, we are inspired to combine the two approaches and transfer the rich results from the lattice literature to Fourier expansion schemes. 

This work is organized as follows. 
In Section \ref{section_scheme}, we present the two components of the cosine expansion lattice scheme, the half-period cosine expansion and the tent transformed lattice. 
We also present the full scheme in this section. 
In Section \ref{section_discussion}, we give further details regarding our current results by providing an alternative formation of its error bound and connecting the scheme to periodic wavelets. 
Numerical experiments are presented in Section \ref{section_numerical_experiment} and we conclude our findings in Section \ref{section_conclusion}.

Before we begin, we mention some conventions used in this work.
We assume all the integrals in the computation to be finite, therefore Fubini's theorem can be applied and we exchange the order of integration without notice. 
The operations $\times$ and $/$ act component-wise when used on vectors. 

Finally, some notation we use in this article:
\begin{itemize}
	\item The natural number set $\mathbb{N} := \{0, 1, 2, \ldots\}$;
	\item The positive integer set $\mathbb{Z}^+ := \{i \in \mathbb{Z}|i>0\}$, similarly for $\mathbb{R}^+$;
	\item The truncated integer set, for $s \in \mathbb{Z}^+$ we write $[s] := \{1, \ldots, s\}$;
	\item The indicator function ${\bf 1}_{\mathcal{D}}: \mathbb{R}^s \rightarrow \{0, 1\}$
		\begin{equation*}
			{\bf 1}_\mathcal{D} ({\bf y}) = 
			\left\{
				\begin{array}{ll}
					1, & \mbox{ if } {\bf y} \in \mathcal{D};\\
					0, & \mbox{ otherwise;}
				\end{array}
			\right.
		\end{equation*}
	\item The ceil function $\lceil \cdot \rceil : \mathbb{R} \rightarrow \mathbb{Z}$, $\lceil y \rceil = \min\{i \in \mathbb{Z}| i\geq y\}$.
\end{itemize}

\section{Cosine Expansion Lattice Scheme} \label{section_scheme}

In this section, we introduce the cosine expansion lattice scheme, whose construction  consists of two parts: a projection of the original integrand on a reproducing kernel space and a numerical integration technique based on a tent-transformed lattice rule.
We will briefly describe the intuition behind our derivation. 

\subsection{The half-period cosine space}

In a similar framework as the COS method from  \cite{fang_2009_novel}, we would like to define a cosine based periodic expansion of function $\mathfrak{f}$ in $\mathbb{R}^s$.
This allows us to connect the expectation problem \eqref{equation_expectation} to a Fourier transform. 

The common practice to transform Equation \eqref{equation_expectation} into a finite problem for computational purposes is restricting the domain of integration to a predefined box $\mathcal{D} := [a_1, b_1] \times \cdots \times [a_s, b_s]$.
This step is justified as long as $\mathcal{D}$ contains the majority of the mass of the measure.
In fact, this is equivalent to replacing the original integrand $\mathfrak{f}({\bf y})$ by $\mathfrak{f}({\bf y}){\bf 1}_{\mathcal{D}}({\bf y})$, or setting the function values outside $\mathcal{D}$ to zero. 

The COS method uses this concept to its advantage by replacing an originally non-periodic one-dimensional function $\mathfrak{f}$ by a cosine series projection that coincides with $\mathfrak{f}$ on the domain $[a, b]$, but is periodic throughout the whole real line.
 
To be precise, the integrand $\mathfrak{f}$ is replaced by\footnote{The notation $\sum{}'$ denotes the first term of the summation is weighted by one half.} $$\bar{\mathfrak{f}}(y) := \sum^{N-1}_{k = 1} {}'\; \widetilde{\mathfrak{f}}_{\cos}(k) \cos \left(k \pi \frac{y-a}{b-a}\right),$$ where $$\widetilde{\mathfrak{f}}_{\cos}(k) :=  \frac{2}{b-a}\int ^b_a \mathfrak{f}(y) \cos \left(k \pi \frac{y-a}{b-a}\right)dx.$$
In this case, the original 1-D expectation is approximated by 
\begin{align*}
	\mathbb{E}[\mathfrak{f}(Y)] 
	& \approx 
	\int_\mathbb{R} \sum^{N-1}_{k = 1} {}'\; \widetilde{\mathfrak{f}}_{\cos}(k) \cos \left(k \pi \frac{y-a}{b-a}\right) \mu(dy) \\
	& = 
	\sum^{N-1}_{k = 1} {}'\; \widetilde{\mathfrak{f}}_{\cos}(k) \int_\mathbb{R} \cos \left(k \pi \frac{y-a}{b-a}\right) \mu(dy) \\
	& = 
	\sum^{N-1}_{k = 1} {}' \;\widetilde{\mathfrak{f}}_{\cos}(k) \Re \left\{\mathcal{F}\left(\frac{k\pi}{b-a}\right)\exp\left(-\imath k \pi\frac{ a}{b-a}\right)\right\}.
\end{align*}
Note that within $[a, b]$, $\bar{\mathfrak{f}}$ is a projection of $\mathfrak{f}$ onto a finite cosine series space and $\bar{\mathfrak{f}}$ converges to $\mathfrak{f}$ in the $L^2([a, b])$ norm, when $N$ tends to infinity. 
However, when considering the function $\bar{\mathfrak{f}}$ on the whole domain $\mathbb{R}$, it is a periodic function and it thus deviates from $\mathfrak{f}$ itself. 
Once again, the error of this transformation is kept under control as the probability mass outside $[a, b]$ is small and we make use of the Fourier transform $\mathcal{F}$, which is typically available, in the approximation. 

The main goal of this work is to apply the same idea in a higher-dimensional setting. 
Thus, we will consider a functional space built on cosine functions and use the series projection of $\mathfrak{f}$ onto such space as our replacement integrand. 
The space we pick is a modification to the half-period cosine space introduced in Section 2.3 of \cite{dick_2014_lattice} and their article serves as an inspiration of the current work.

We define an inner product $<\cdot, \cdot>_{\mathcal{K}^{\cos}_{\alpha, {\boldsymbol{\gamma}}, s}(\mathcal{D})}$ as $$<\mathfrak{f}, \mathfrak{g}>_{\mathcal{K}^{\cos}_{\alpha, {\boldsymbol{\gamma}}, s}(\mathcal{D})} = \sum_{{\bf k} \in \mathbb{N}^s}\widetilde{\mathfrak{f}}_{\cos}({\bf k})\widetilde{\mathfrak{g}}_{\cos}({\bf k})r^{-1}_{\alpha, \boldsymbol{\gamma}, s}({\bf k}),$$ for some real number $\alpha > 1/2$ and vector $\boldsymbol{\gamma} \in (\mathbb{R}^+)^s$, where the multi-dimensional cosine coefficients  are given by $$\int_{[0, 1]^s}\mathfrak{f}({\bf y}\times({\bf b-a}) + {\bf a}) 2^{|{\bf k}|_0/2}\prod^s_{j=1}\cos (\pi k_j y_j) d{\bf y}, \; \mbox{for } {\bf k} = (k_1, \cdots, k_s) \in \mathbb{N}^s.$$
Here we define $|{\bf k}|_0 := \# \{j \in [s] : k_j \neq 0\}$ to be the number of non-zero components in ${\bf k }$. 
Note that only the portion of $\mathfrak{f}$ within the predefined domain $\mathcal{D}$ is used here. 

For $\alpha > 1/2, k \in \mathbb{Z}$ and $\gamma > 0$, the one-dimensional $r$ function is defined as:
\begin{equation*}
	r_{\alpha, \gamma}(k):= \left\{
	\begin{array}{ll}
	1 & \mbox{if } k=0;\\
	\gamma |k|^{-2\alpha} & \mbox{if } k \neq 0,
	\end{array}\right.
\end{equation*}
and the multidimensional $r$ function is set as, 
\begin{equation*}
	r_{\alpha, \boldsymbol{\gamma}, s}({\bf k}):=\prod^s_{j=1}r_{\alpha, \gamma_j}(k_j),
\end{equation*}
for ${\bf k} = (k_1, \ldots, k_s)\in \mathbb{Z}^s$ and ${\boldsymbol \gamma} = (\gamma_1, \ldots, \gamma_s)\in (\mathbb{R}^+)^s$.
The $r$ function is introduced to the norm here to assess the decay rate of the cosine coefficients, as it is closely related to the approximation error. 

We define the corresponding norm as $\sqrt{<\mathfrak{f}, \mathfrak{f}>_{\mathcal{K}^{\cos}_{\alpha, {\boldsymbol{\gamma}}, s}(\mathcal{D})}}$ and denote it by $||\mathfrak{f}||_{\mathcal{K}^{\cos}_{\alpha, {\boldsymbol{\gamma}}, s}(\mathcal{D})}$.
In particular, we have $$||\mathfrak{f}||_{\mathcal{K}^{\cos}_{\alpha, {\boldsymbol{\gamma}}, s}(\mathcal{D})}^2 = \sum_{{\bf k}\in \mathbb{N}^s}\frac{|\widetilde{\mathfrak{f}}_{\cos}({\bf k})|^2}{r_{\alpha, {\bf \gamma}, s}({\bf k})} = \sum_{{\bf h}\in \mathbb{Z}^s}2^{-|{\bf h}|_0}\frac{|\widetilde{\mathfrak{f}}_{\cos}({\bf h})|^2}{r_{\alpha, {\bf \gamma}, s}({\bf h})}.$$ 
The dummy variable is changed from ${\bf k}$ to ${\bf h}$ here in preparation for future computations. 
We denote any function $\mathfrak{f}$ such that $||\mathfrak{f}||_{\mathcal{K}^{\cos}_{\alpha, {\boldsymbol{\gamma}}, s}(\mathcal{D})} < \infty$ as $\mathfrak{f} \in \mathcal{H}_{\mathcal{K}^{\cos}_{\alpha, {\boldsymbol{\gamma}}, s}(\mathcal{D})}$.

The half-period cosine space is an example of a reproducing kernel Hilbert space with the corresponding reproducing kernel,
\begin{align*}
	\mathcal{K}^{\cos}_{\alpha, \boldsymbol{\gamma}, s}(\mathcal{D}, {\bf x}, {\bf y}) 
	:= & 
	\prod^s_{j=1}  \left(\sum _{k_j \in \mathbb{Z}} r_{\alpha, \gamma_j}(k_j)\cos \left(\pi k_j \frac{x_j- a_j}{b_j - a_j}\right) e^{\imath \pi k_j \frac{y_j -a_j}{b_j - a_j}}\right) \\
	= & \sum_{{\bf k} \in \mathbb{Z}^s}r_{\alpha, \boldsymbol{\gamma}, s}({\bf k})\left(\prod^s_{j=1}\cos \left(\pi k_j \frac{x_j - a_j}{b_j - a_j}\right)\right) e^{\imath \pi {\bf k}\cdot{\bf \frac{y-a}{b-a}}},
\end{align*}
and for the one-dimensional version, 
\begin{align*}
	\mathcal{K}_{\alpha, \gamma}^{\cos}(\mathcal{D}, x, y) 
	:= & 
	1 + \sum^\infty_{k = 1}r_{\alpha, \gamma}(k) \sqrt{2}\cos \left(\pi k \frac{x-a}{b-a}\right) \sqrt{2} \cos \left(\pi k \frac{y-a}{b-a}\right) \\
	= & 
	1 + \sum^\infty_{k=1}r_{\alpha, \gamma}(k)\cos\left(\pi k \frac{x-a}{b-a}\right) \left(e^{\imath \pi k \frac{y-a}{b-a}} + e^{-\imath \pi k \frac{y-a}{b-a}}\right)\\
	= &
	\sum _{k \in \mathbb{Z}} r_{\alpha, \gamma}(k)\cos \left(\pi k \frac{x-a}{b-a}\right) e^{\imath \pi k \frac{y-a}{b-a}}.\\
\end{align*}
For any ${\bf y} \in \mathcal{D}$ and $\mathfrak{f} \in \mathcal{H}_{\mathcal{K}^{\cos}_{\alpha, \boldsymbol{\gamma}, s} (\mathcal{D})}$, we have the reproducing property, $$\mathfrak{f}({\bf y}) = <\mathfrak{f}, \mathcal{K}^{\cos}_{\mathcal{D}, \alpha, \boldsymbol{\gamma}, s}(\cdot, y)>_{\mathcal{K}^{\cos}_{\alpha, {\boldsymbol{\gamma}}, s}(\mathcal{D})}.$$
Readers are referred to \cite{Aronszain_1950_theory} and \cite{dick_2010_digital} for further information on reproducing kernel Hilbert spaces.

In this work, we use an alternative kernel which drops the $r$ function. 
It is defined as 
\begin{align*}
	\mathcal{K}^{\cos}_{s}({\bf x}, {\bf y}) 
	:= & 
	\prod^s_{j=1}  \left(\sum _{k_j \in \mathbb{Z}} \cos \left(\pi k_j \frac{x_j- a_j}{b_j - a_j}\right) e^{\imath \pi k_j \frac{y_j -a_j}{b_j - a_j}}\right) 
	= 
	\sum_{{\bf k} \in \mathbb{Z}^s} \left(\prod^s_{j=1}\cos \left(\pi k_j \frac{x_j - a_j}{b_j - a_j}\right)\right) e^{\imath \pi {\bf k}\cdot{\bf \frac{y-a}{b-a}}}.
\end{align*}
We suppress the $\mathcal{D}$ part here to simplify our notation.  

Using the reproducing property, we have the following equation: 
\begin{equation}
	\mathfrak{f}({\bf y}) = \sum_{{\bf k}\in \mathbb{N}^s} \widetilde{\mathfrak{f}}_{\cos}({\bf k})(\sqrt{2})^{|{\bf k}|_0}\prod_{j=1}^{s}\cos\left(\pi k_i \frac{y_i -a_i}{b_i - a_i}\right),
	\label{equation_reproducing_property}
\end{equation}
for any  $\mathfrak{f} \in \mathcal{H}_{\mathcal{K}^{\cos}_{\alpha, {\boldsymbol{\gamma}}, s}(\mathcal{D})}$ and ${\bf y} \in \mathcal{D}$. 
We use the expansion at the right-hand side of \eqref{equation_reproducing_property} as the replacement integrand in Equation \eqref{equation_expectation}, denoted by $\bar{\mathfrak{f}}$. 
\begin{definition}[Half-period Cosine Expansion]
	 For any given function  $\mathfrak{f}: \mathbb{R}^s \rightarrow \mathbb{R}$ and given domain $\mathcal{D} \subset \mathbb{R}^s$, the half-period cosine expansion, $\bar{\mathfrak{f}}: \mathbb{R}^s \rightarrow \mathbb{R}$, is defined as 
	\begin{equation*}
	\bar{\mathfrak{f}}({\bf y}) := \sum_{{\bf k}\in \mathbb{N}^s} \widetilde{\mathfrak{f}}_{\cos}({\bf k})(\sqrt{2})^{|{\bf k}|_0}\prod_{j=1}^{s}\cos\left(\pi k_i \frac{y_i -a_i}{b_i -a_i}\right), \label{equation_half-period_cosine_expansion}
	\end{equation*}
	where the cosine coefficients are defined as
	\begin{equation*}
	\widetilde{\mathfrak{f}}_{\cos}({\bf k}) := \int_{[0, 1]^s} \mathfrak{f}({\bf y}\times({\bf b-a}) + {\bf a}) 2^{|{\bf k}|_0/2}\prod^s_{j=1}\cos (\pi k_j y_j) d{\bf y}.
	\end{equation*}
\end{definition}

\subsection{Lattice rule approximations} \label{section_lattice_integration}

There are essentially two drawbacks when extending the COS method to higher dimensions. 
First, the cosine coefficient $\widetilde{\mathfrak{f}}_{\cos}$ may be difficult to calculate, especially in a recurring situation. 
In our previous work \cite{chau_2018_wavelet}, we aimed to remedy this shortcoming by adopting a wavelet basis such that we can approximate Equation \eqref{equation_expectation} as a weighted sum of local values, in the form of $$\mathbb{E}[\mathfrak{f}(Y)]\approx \frac{1}{N}\sum_{n=1-N}^N\mathfrak{f}\left(\frac{r}{2^{\mathtt{m}}}\right)\mathbb{E}[\varphi_{N, n}(Y)],$$ with $\varphi$ being the wavelet basis functions. 
We will compare our current scheme to that work in Section \ref{section_cosine_wavelets}.

The second point of concern is that if we simply apply the tensor product of cosine spaces in a higher-dimensional case, as we have done in the last section, we will face the curse of dimensionality. 

Here, we aim to address the above two issues by constructing quadrature rules for the integration of
\begin{equation}
	 \int_{\mathbb{R}^s} \bar{\mathfrak{f}}({\bf y})\mu(d {\bf y}). \label{equation_cosine_integration}
\end{equation}
In particular, the approximant takes the specific form,
\begin{equation}
	\int_{\mathbb{R}^s} \frac{1}{N} \sum^{N-1}_{n=0} 
	\mathfrak{f}\left({\bf p}_n\right) 
	\mathcal{K}^{\cos}_s\left({\bf p}_n, {\bf y}\right) \mu (d{\bf y}), \label{equation_quadrature_rule}
\end{equation}
where $\mathit{P}:= \{{\bf p}_0, \ldots, {\bf p}_{N-1}\}$ is some predetermined quadrature points set.

For the half-period cosine space that we adopted in the previous section, Dick et. al. showed in \cite{dick_2014_lattice} that a combination of rank-1 lattice rules and tent transformations converges well in the context of quasi-Monte Carlo methods. 
We will now introduce the quadrature points proposed in \cite{dick_2014_lattice} for the half-period cosine space, which we should adapt for the approximation of Equation \eqref{equation_quadrature_rule}.

A lattice point set with $N\geq 2$ points and generating vector ${\bf g}\in [N-1]^s$ is given by $$\mathit{P}({\bf g}, N):= \left\{\left\{\frac{n {\bf g}}{N}\right\} : 0 \leq n < N\right\}.$$
The {\it tent transformation}, $\phi: [0, 1] \rightarrow [0, 1]$ is defined as $$\phi(x) := 1 - |2x-1|,$$
and the higher-dimensional version $\boldsymbol{\phi}: [0, 1]^s \rightarrow [0, 1]^s$ is obtained by applying the function component-wise. 
The tent-transformed lattice point set in \cite{dick_2014_lattice} is given by $$\mathit{P}_{\boldsymbol{\phi}}({\bf g}, N) :=\left\{\boldsymbol{\phi}\left(\left\{\frac{n{\bf g}}{N}\right\}\right):0\leq n <N \right\}.$$
However, since we have to apply the quadrature points on a general box $\mathcal{D}$, instead of just on $[0, 1]^s$, we have to transform the lattice rules: 
\begin{align*}
	\mathit{P}_{\boldsymbol{\phi}}({\bf g}, N, \mathcal{D}) 
	= &
	\{ {\bf p}_{\boldsymbol{\phi}, 0}, \cdots, {\bf p}_{\boldsymbol{\phi}, N-1}\} 
	:=
	\left\{\boldsymbol{\phi}\left(\left\{\frac{n{\bf g}}{N}\right\}\right) \times {\bf (b - a)} + {\bf a}:0\leq n <N \right\}.
\end{align*}

We wish to quantify the integration error of applying Equation \eqref{equation_quadrature_rule} to approximate the expectation in the form of equation \eqref{equation_cosine_integration} for functions in the half-period cosine space $\mathcal{H}_{\mathcal{K}^{\cos}_{\alpha, \boldsymbol{\gamma}, s}(\mathcal{D})}$.
In addition to the above condition, we also enforce some convergence requirements on the cosine transform of the measure $\mu$. 

\begin{theorem}\label{theorem_lattice}
	Consider any function $\mathfrak{f} \in \mathcal{H}_{\mathcal{K}^{\cos}_{\alpha, \boldsymbol{\gamma}, s}(\mathcal{D})}$ and any probability measure $\mu$ such that
	\begin{equation*}
	\left(\int_{\mathbb{R}^s}\left(\prod^s_{j=1}\cos \left(\pi k_j \frac{y_j - a_j}{b_j -a_j}\right) \right)
	\mu(d {\bf y})\right)^2 \leq r_{\beta, \boldsymbol{\rho}, s}({\bf k}), \label{equation_measure_decay_assumption}
	\end{equation*}
	for some real number $\beta > s$ and vector $\boldsymbol{\rho} \in (\mathbb{R}^+)^s$.
	Furthermore, we assume that $\beta - \alpha > 1/2$, namely, the cosine transform of the measure $\mu$ decays at least algebraically and quicker than the cosine coefficients of $\mathfrak{f}$.
	The error for the numerical integration using the tent-transformed lattice rule on $\mathcal{D}$ is then bounded by
	\begin{align*}
		\epsilon_2(\mathfrak{f}, \mu, \mathit{P}_{\boldsymbol{\phi}}({\bf g}, N, \mathcal{D})) 
		:= &
		\left|\int_{\mathbb{R}^s} 
		\frac{1}{N} \sum^{N-1}_{n=0} 
		\bar{\mathfrak{f}}\left({\bf p}_{\boldsymbol{\phi}, n}\right) 
		\mathcal{K}^{\cos}_s\left({\bf p}_{\boldsymbol{\phi}, n}, {\bf y}\right) \mu (d{\bf y})
		- \int_{\mathbb{R}^s} \bar{\mathfrak{f}}({\bf y}) \mu(d {\bf y})\right| \\
		\leq &
		\left(\sum_{{\bf h}\in L^\bot\backslash\{{\bf 0}\}}r_{\alpha,\boldsymbol{\gamma}, s}({\bf h})\right)^\frac{1}{2}\left(\prod^s_{i=1}C_i\right)^{\frac{1}{2}}
		||\mathfrak{f}||_{\mathcal{K}^{\cos}_{\alpha, \boldsymbol{\gamma}, s}(\mathcal{D})}
	\end{align*}
	for some constant $C_i$, where	$L^\bot := \{{\bf h} \in \mathbb{Z}^s : {\bf h}\cdot{\bf g} \equiv 0 \; (\bmod \;N)\}$ is the dual lattice. 
\end{theorem}
\begin{proof}
	The general flow of this proof follows the proof of Theorem 2 in \cite{dick_2014_lattice}.
	Let $\mathfrak{f} \in \mathcal{H}_{\mathcal{K}^{\cos}_{\alpha, \boldsymbol{\gamma}, s}(\mathcal{D})}$ and $\bar{\mathfrak{f}}$ be its half-period cosine expansion.
	
	For any $k\in \mathbb{N}$, we have $$\cos(\pi k \phi(y)) = \cos(2\pi k y)\;\mbox{ for all } y\in [0,1],$$ and hence
	\begin{align*}
		\bar{\mathfrak{f}} \left( {\bf p }_{\boldsymbol{\phi}, n}
		\right)
		= &
		\sum_{{\bf k}\in \mathbb{N}^s} \widetilde{\mathfrak{f}}_{\cos}({\bf k})
		(\sqrt{2})^{|{\bf k}|_0} 
		\prod_{j=1}^s\cos\left(\pi k_j \phi\left(\left\{\frac{ng_j}{N} \right\}\right) \right)\\
		= &
		\sum_{{\bf k}\in \mathbb{N}^s} \widetilde{\mathfrak{f}}_{\cos}({\bf k})
		(\sqrt{2})^{|{\bf k}|_0} 
		\prod_{j=1}^s\cos\left(2 \pi k_j \frac{ng_j}{N} \right)\\
		= & 
		\sum_{{\bf h}\in \mathbb{Z}^s} (\sqrt{2})^{-|{\bf h}|_0}
		    \widetilde{\mathfrak{f}}_{\cos}({\bf h}) e^{2\pi \imath n ({\bf h}\cdot{\bf g})/N}.
	\end{align*}
	The reproducing kernel can be rewritten as 
	\begin{align*}
	\mathcal{K}^{\cos}_{s} \left( {\bf p}_{\boldsymbol{\phi}, n}, {\bf y }\right)
	=&
	\sum_{{\bf k}\in\mathbb{Z}^s}
	\left(\prod^s_{j=1}\cos \left(\pi k_j \frac{y_j - a_j}{b_j -a_j}\right) \right)
	e^{2\pi \imath n ({\bf k}\cdot{\bf g})/N}.
	\end{align*}
	
	Therefore, we obtain 
	\begin{align}
		& 
		\int_{\mathbb{R}^s} 
		\frac{1}{N} \sum^{N-1}_{n=0} 
		\bar{\mathfrak{f}}\left({\bf p}_{\boldsymbol{\phi}, n}\right) 
		\mathcal{K}^{\cos}_s\left({\bf p}_{\boldsymbol{\phi}, n}, {\bf y}\right) \mu (d{\bf y})\nonumber\\
		= & 
		\int_{\mathbb{R}^s}
		\frac{1}{N}\sum^{N-1}_{n=0}
		\sum_{{\bf h}\in\mathbb{Z}^s}(\sqrt{2})^{-|{\bf h}|_0}
		\widetilde{\mathfrak{f}}_{\cos}({\bf h})e^{2\pi \imath n ({\bf h}\cdot{\bf g})/N}
			\sum_{{\bf k}\in\mathbb{Z}^s}
		\left(\prod^s_{j=1}\cos \left(\pi k_j \frac{y_j - a_j}{b_j -a_j}\right) \right)
		e^{2\pi \imath n ({\bf k}\cdot{\bf g})/N}
		\mu(d{\bf y})\nonumber\\
		= & 
		\int_{\mathbb{R}^s}
		\sum_{{\bf h}\in\mathbb{Z}^s}\sum_{{\bf k}\in\mathbb{Z}^s}
		(\sqrt{2})^{-|{\bf h}|_0}\widetilde{\mathfrak{f}}_{\cos}({\bf h})
		\left(\prod^s_{j=1}\cos \left(\pi k_j \frac{y_j - a_j}{b_j -a_j}\right) \right)
		\frac{1}{N}\sum^{N-1}_{n=0}
		e^{2\pi \imath n [({\bf h}+{\bf k})\cdot {\bf g}]/N}
		\mu(d{\bf y}).\nonumber
	\end{align}
	The sum $\frac{1}{N}\sum^{N-1}_{n=0} e^{2\pi \imath n [({\bf h}+{\bf k})\cdot {\bf g}]/N}$ is a character sum over the group $\mathbb{Z}/N\mathbb{Z}$, which is equal to one if ${\bf (h+k)}\cdot {\bf g}$ is a multiple of $N$ and zero otherwise.
	From this, we get 
	\begin{align}
		&
		\int_{\mathbb{R}^s} 
		\frac{1}{N} \sum^{N-1}_{n=0} 
		\bar{\mathfrak{f}}\left({\bf p}_{\boldsymbol{\phi}, n}\right) 
		\mathcal{K}^{\cos}_s\left( {\bf p}_{\boldsymbol{\phi}, n}, {\bf y}\right) \mu (d{\bf y})
		- \int_{\mathbb{R}^s} \bar{\mathfrak{f}}({\bf y}) \mu(d {\bf y}) \nonumber\\
		= &
		\int_{\mathbb{R}^s}
		\sum_{{\bf h}\in L^\bot\backslash\{{\bf 0}\}}\sum_{{\bf k}\in \mathbb{Z}^s}
		(\sqrt{2})^{-|{\bf h}-{\bf k}|_0}\widetilde{\mathfrak{f}}_{\cos}({\bf h}-{\bf k})
		\left(\prod^s_{j=1}\cos \left(\pi k_j \frac{y_j - a_j}{b_j -a_j}\right) \right)
		\mu(d{\bf y}). \label{equation_error_expression}
	\end{align}
	Based on this formula and an application of the Cauchy-Schwarz inequality, we obtain
	\begin{align}
	&
	\left|\int_{\mathbb{R}^s} 
	\frac{1}{N} \sum^{N-1}_{n=0} 
	\bar{\mathfrak{f}}\left({\bf p}_{\boldsymbol{\phi}, n}\right) 
	\mathcal{K}^{\cos}_s\left( {\bf p}_{\boldsymbol{\phi}, n}, {\bf y}\right) \mu (d{\bf y})
	- \int_{\mathbb{R}^s} \bar{\mathfrak{f}}({\bf y}) \mu(d {\bf y})\right| \nonumber\\
	= & 
	\left|\sum_{{\bf h}\in L^\bot\backslash\{{\bf 0}\}}\sum_{{\bf k}\in\mathbb{Z}^s}
	(\sqrt{2})^{-|{\bf h}-{\bf k}|_0}\widetilde{\mathfrak{f}}_{\cos}({\bf h}-{\bf k})
	\int_{\mathbb{R}^s}\left(\prod^s_{j=1}\cos \left(\pi k_j \frac{y_j - a_j}{b_j -a_j}\right) \right)
	\mu(d {\bf y})\right| \nonumber\\
	\leq & 
	\left( \sum_{{\bf h}\in L^\bot\backslash\{{\bf 0}\}} 
	\left(\sum_{{\bf k}\in\mathbb{Z}^s}
	(\sqrt{2})^{-|{\bf h}-{\bf k}|_0}\widetilde{\mathfrak{f}}_{\cos}({\bf h}-{\bf k})
	\int_{\mathbb{R}^s}\left(\prod^s_{j=1}\cos \left(\pi k_j \frac{y_j - a_j}{b_j -a_j}\right) \right)
	\mu(d {\bf y})
	\right)^2
	\right)^\frac{1}{2}\nonumber\\
	\leq & \left( 
	\sum_{{\bf h}\in L^\bot\backslash\{{\bf 0}\}}
	\left(\sum_{{\bf k}\in\mathbb{Z}^s}
	\frac{
		2^{-|{\bf h}-{\bf k}|_0}|\widetilde{\mathfrak{f}}_{\cos}({\bf h}-{\bf k})|^2}
	{r_{\alpha, \boldsymbol{\gamma}, s}({\bf h}-{\bf k})}\right) 
	\left(\sum_{{\bf k}\in\mathbb{Z}^s}
	r_{\alpha, \boldsymbol{\gamma}, s}({\bf h}-{\bf k})
	\left(
	\int_{\mathbb{R}^s}\left(\prod^s_{j=1}\cos \left(\pi k_j \frac{y_j - a_j}{b_j -a_j}\right) \right)
	\mu(d {\bf y})
	\right)^2\right)
	\right)^\frac{1}{2}\nonumber\\
	\leq & 
	\left(\sum_{{\bf h}\in L^\bot\backslash\{{\bf 0}\}} 
	\sum_{{\bf k}\in\mathbb{Z}^s}
	r_{\alpha, \boldsymbol{\gamma}, s}({\bf h}-{\bf k})
	\left(
	\int_{\mathbb{R}^s}\left(\prod^s_{j=1}\cos \left(\pi k_j \frac{y_j - a_j}{b_j -a_j}\right) \right)
	\mu(d {\bf y})
	\right)^2
	\right)^\frac{1}{2}
	||\mathfrak{f}||_{\mathcal{K}^{\cos}_{\alpha, \boldsymbol{\gamma}, s}(\mathcal{D})} \nonumber\\
	=: & ||\mathfrak{f}||_{\mathcal{K}^{\cos}_{\alpha, \boldsymbol{\gamma}, s}(\mathcal{D})}
	\left(\sum_{{\bf h}\in L^\bot\backslash\{{\bf 0}\}}I({\bf h})\right)^\frac{1}{2}. \label{equation_cauchy_schwarz_error}
	\end{align}

	Making use of the smoothness assumption on the probability measure, we find
	\begin{align}
	I({\bf h}) 
	= &
	\sum_{{\bf k}\in\mathbb{Z}^s}
	r_{\alpha, \boldsymbol{\gamma}, s}({\bf h}-{\bf k})
	\left(
	\int_{\mathbb{R}^s}\left(\prod^s_{j=1}\cos \left(\pi k_j \frac{y_j - a_j}{b_j -a_j}\right) \right)
	\mu(d {\bf y})
	\right)^2 \nonumber\\
	\leq &
	\sum_{{\bf k}\in\mathbb{Z}^s}
	r_{\alpha, \boldsymbol{\gamma}, s}({\bf h}-{\bf k})
	r_{\beta, \boldsymbol{\rho}, s}({\bf k})
	\nonumber \\
	= &
	\sum_{{\bf k}\in\mathbb{Z}^s}
	\prod_{j=1}^{s}
	r_{\alpha, \gamma_j}(h_j-k_j)
	r_{\beta, \rho_j}(k_j) \nonumber \\
	= &
	\prod_{j=1}^{s}
	\sum_{k_j = -\infty}^\infty
	r_{\alpha, \gamma_j}(h_j-k_j)
	r_{\beta, \rho_j}(k_j) \label{equation_sum_of_coverging}
	=: \prod_{j=1}^{s} I_j(h_j).
	\end{align}
	
	In order to control the error, we need to control the sum $I_j$ in Equation \eqref{equation_sum_of_coverging} for all $h_j \in \mathbb{Z}$. 
	We have three different cases to consider.
	\paragraph{Case 1: $h_j=0.$} 
	In this case, 
	\begin{align*}
	I_j(0) = &
	\sum_{k_j = -\infty}^\infty
	r_{\alpha, \gamma_j}(-k_j)
	r_{\beta, \rho_j}(k_j) 
	=
	1 +  2\gamma_j\rho_j\sum^\infty_{k_j=1} \frac{1}{(k_j)^{2\alpha + 2\beta}} = 
	1 +  2\gamma_j\rho_j\zeta(2\alpha + 2\beta),
	\end{align*}
	in which $\zeta$ denotes the Riemann zeta function.
	
	\paragraph{Case 2: $h_j > 0.$}
	In this case, 
	\begin{align*}
	I_j(h_j) = &
	\sum_{k_j = -\infty}^\infty
	r_{\alpha, \gamma_j}(h_j-k_j)
	r_{\beta, \rho_j}(k_j) \\
	= & 
	\sum^{\infty}_{k_j = 1} r_{\alpha, \gamma_j}(h_j+k_j)
	r_{\beta, \rho_j}(k_j) 
	+ r_{\alpha, \gamma_j}(h_j) 
	+ \sum^{h_j-1}_{k_j = 1} r_{\alpha, \gamma_j}(h_j-k_j)
	r_{\beta, \rho_j}(k_j)
	+ r_{\beta, \rho_j}(h_j) \\
	&
	+ \sum^\infty_{k_j = 1}r_{\alpha, \gamma_j}(k_j)
	r_{\beta, \rho_j}(h_j + k_j)\\
	\leq &
	\gamma_j |h_j|^{-2\alpha} \sum^\infty_{k_j = 1} \rho_j |k_j|^{-2\beta} + r_{\alpha, \gamma_j}(h_j) 
	 + \sum^{h_j-1}_{k_j = 1} \gamma_j |h_j - k_j|^{-2\alpha}\rho_j |k_j|^{-2\beta} + \frac{\rho_j}{\gamma_j} r_{\alpha, \gamma_j}(h_j) \\
	&
	 + \frac{\rho_j}{\gamma_j}r_{\alpha, \gamma_j}(h_j)\sum^\infty_{k_j =1} 
	r_{\alpha, \gamma_j}(k_j)\\
	\leq & 
	\rho_j \zeta(2\beta) r_{\alpha, \gamma_j}(h_j)+ r_{\alpha, \gamma_j}(h_j) 
	 + \gamma_j |h_j|^{-2\alpha}2^{2\alpha} \sum^{h_j-1}_{k_j = 1} \rho_j |k_j|^{-2(\beta-\alpha)} + \frac{\rho_j}{\gamma_j} r_{\alpha, \gamma_j}(h_j) \\
	&
	 + \rho_j r_{\alpha, \gamma_j}(h_j)\sum^\infty_{k_j =1} 
	|k_j|^{-2 \alpha }\\
	\leq & 
	\rho_j \zeta(2\beta) r_{\alpha, \gamma_j}(h_j)+ r_{\alpha, \gamma_j}(h_j) + 2^{2\alpha} \rho_j \zeta(2(\beta-\alpha)) r_{\alpha, \gamma_j} (h_j)+ \frac{\rho_j}{\gamma_j} r_{\alpha, \gamma_j}(h_j) + \rho_j \zeta(2\alpha) r_{\alpha, \gamma_j}(h_j).
	\end{align*}
	The above inequality simply follows from the definition and the orderings $|h_j + k_j|^{-1} < |h_j|^{-1}$, $|h_j|^{-\beta} < |h_j|^{-\alpha}$ and $(h_j -1) \times 1 < (h_j -2) \times 2 < \cdots \left(\frac{h_j}{2}\right)$.
	Note that we use the convention $\sum^0_{k_j = 1} =0$ here, so the inequality also holds for $h_j = 1$. 
	
	\paragraph{Case 3: $h_j < 0.$} Finally, in this case, 
	\begin{align*}
	I_j(h_j) = &
	\sum_{k_j = -\infty}^\infty
	r_{\alpha, \gamma_j}(h_j-k_j)
	r_{\beta, \rho_j}(k_j) \\
	= & 
	\sum_{k_j = 1}^{\infty}
	r_{\alpha, \gamma_j}(k_j)
	r_{\beta, \rho_j}(k_j-h_j) 
	+ r_{\beta, \rho_j}(h_j)
	+\sum_{k_j = 1}^{-h_j-1}
	r_{\alpha, \gamma_j}(h_j+k_j)
	r_{\beta, \rho_j}(k_j) 
	+ r_{\alpha, \gamma_j}(h_j)\\
	& + \sum_{k_j = 1}^\infty
	r_{\alpha, \gamma_j}(h_j-k_j)
	r_{\beta, \rho_j}(k_j)\\
	\leq & 
	\rho_j \zeta(2 \alpha) r_{\alpha, \gamma_j}(h_j) + \frac{\rho_j}{\gamma_j} r_{\gamma_j, \alpha}(h_j) \\
	&
	+ \sum_{k_j = 1}^{-h_j-1} \gamma_j |h_j + k_j|^{-2\alpha} \rho_j |k_j|^{-2\beta}
	+ r_{\alpha, \gamma_j}(h_j) + \rho_j \zeta(2 \beta) r_{\alpha, \gamma_j}(h_j)\\
	\leq & 
	\rho_j (\zeta(2\alpha) + \zeta(2\beta) )r_{\alpha, \gamma_j}(h_j) + \frac{\rho_j}{\gamma_j} r_{\gamma_j, \alpha}(h_j) + r_{\alpha, \gamma_j}(h_j) + 2^{2\alpha}r_{\alpha, \gamma_j} (h_j)\rho_j \zeta(2(\beta-\alpha)).
	\end{align*}
	The derivation of Case 3 is similar to Case 2. 
	The only difference is that we have the following orderings: $|-h_j+k_j| > |h_j|$ for $k_j > 0$ and $(-h_j -1)\times 1 < (-h_j-2) \times 2 < \cdots < \left(\frac{h_j}{2}\right)^2$.
	
	Finally, let $C_j = \max \{1 +  2\gamma_j\rho_j\zeta(2\alpha + 2\beta), 1 + \rho_j\left(\zeta(2\alpha) + \zeta(2\beta) + \frac{1}{\gamma_j} + 2^{2\alpha} \zeta(2(\beta-\alpha))\right)\}$, which is independent of $h_j$. 
	Then, 
	\begin{align*} 
	&
	\left|\int_{\mathbb{R}^s}
	\sum_{{\bf h}\in L^\bot\backslash\{{\bf 0}\}}\sum_{{\bf k}\in\mathbb{Z}^s}
	(\sqrt{2})^{-|{\bf h}-{\bf k}|_0}\widetilde{f}_{\cos}({\bf h}-{\bf k})
	\left(\prod^s_{j=1}\cos \left(\pi k_j \frac{y_j - a_j}{b_j -a_j}\right) \right)
	\mu(d{\bf y}) \right|\\
	\leq &  
	\left(\sum_{{\bf h}\in L^\bot\backslash\{{\bf 0}\}}\prod_{j =1}^{s}C_jr_{\alpha, \gamma_j}(h_j)\right)^\frac{1}{2} ||\mathfrak{f}||_{\mathcal{K}^{\cos}_{\alpha, \boldsymbol{\gamma}, s}(\mathcal{D})}
	= \left(\sum_{{\bf h}\in L^\bot\backslash\{{\bf 0}\}}r_{\alpha,\boldsymbol{\gamma}, s}({\bf h})\right)^\frac{1}{2}\left(\prod^s_{j=1}C_j\right)^{\frac{1}{2}}
	||\mathfrak{f}||_{\mathcal{K}^{\cos}_{\alpha, \boldsymbol{\gamma}, s}(\mathcal{D})}.
	\end{align*}
	
	In this proof, we break down the sum in Equation \eqref{equation_cauchy_schwarz_error} and derive a bound for each dual lattice point ${\bf h}$ along each direction, namely, the $I_j(h_j)$ terms in \eqref{equation_sum_of_coverging}. 
	By considering the three possible cases for $h_d$, positive, negative and zero, we provide a bound in each case. 
	It is necessary to separate the three cases as we have to identify the section where both $|h_j -k_j|$ and $|k_j|$ are smaller than  $|h_j|$ when $k_j$ is moving along the real line and apply a separated bound on these sections. 
	
	Taking the maximum out of the three cases and putting $I_j(h_j)$ back into the sum in Equation \eqref{equation_quadrature_rule} finishes the proof.
\end{proof}

\begin{remark} \label{remark_uniform_distribution}
	The integrals of the form, $\int_{[0,1]^s}\mathfrak{f}({\bf y})d{\bf y},$ can be seen as special cases of Equation \eqref{equation_cosine_integration} where the probability measure is an uniform distribution over $\mathcal{D} = [0, 1]^s$.
	In this case, 
	\begin{align*}
		\int_{\mathbb{R}^s}\left(\prod^s_{j=1}\cos \left(\pi k_j \frac{y_j - a_j}{b_j -a_j}\right) \right)
		\mu(d {\bf y})
		= &
		\int_{[0, 1]^s}\left(\prod^s_{j=1}\cos \left(\pi k_j y_j \right) \right)
		d {\bf y}\\
		= &
		\left\{\begin{array}{ll}
		1 & \mbox{ if } k \equiv {\bf 0}; \\
		0 & \mbox{ otherwise.}
		\end{array}\right.\
	\end{align*}
	Therefore, we may simplify the square of the term in Equation \eqref{equation_cauchy_schwarz_error}, as follows
	\begin{align*}
		& \sum_{{\bf h}\in L^\bot\backslash\{{\bf 0}\}} 
		\left(\sum_{{\bf k}\in\mathbb{Z}^s}
		(\sqrt{2})^{-|{\bf h}-{\bf k}|_0}\widetilde{\mathfrak{f}}_{\cos}({\bf h}-{\bf k})
		\int_{\mathbb{R}^s}\left(\prod^s_{j=1}\cos \left(\pi k_j \frac{y_j - a_j}{b_j -a_j}\right) \right)
		\mu(d {\bf y})
		\right)^2\\
		= &
		\sum_{{\bf h}\in L^\bot\backslash\{{\bf 0}\}} 
		(\sqrt{2})^{-|{\bf h}|_0}\widetilde{\mathfrak{f}}_{\cos}({\bf h}),
	\end{align*}
	which is the same as the right-hand side of Equation (6) in \cite{dick_2014_lattice}.
	Therefore, our numerical integration can be seen as an extension of the Quasi-Monte Carlo (QMC) rules in \cite{dick_2014_lattice}.
\end{remark}

The above proof is rather rough, note the $2^{2\alpha}$ term in the last two cases. 
However, it suggests we can still relate the integration error to the smoothness of the target function $\mathfrak{f}$ in terms of the cosine coefficients. 
The error can be controlled by the sum $\left(\sum_{{\bf h}\in L^\bot\backslash\{{\bf 0}\}}r_{\alpha,\boldsymbol{\gamma}, s}({\bf h})\right)^{1/2}$. 
Thus all the results for the worst-case error for the QMC integration in the half-period cosine space using tent-transformed lattice rules can be used here.
This opens the door for applying the generating vector for extensible lattice rules from \cite{hickernell_2012_weighted} and other results from the QMC literature to the half-period cosine expansion scheme.
We believe this result is a promising starting point for the study of lattice expansions for higher-dimension numerical integration with general finite measure. 

\subsection{Full Approximation Schemes and Errors}

Now we can introduce the full approximation scheme. 
The final obstacle lies in the reproducing kernel, which is defined as an infinite sum. 
Instead of calculating the full sum, a truncated version is used in the actual algorithm.
We define 
\begin{align}
\mathcal{K}^{\cos}_{s, K}({\bf x}, {\bf y}) 
:= &
\sum_{{\bf k} \in \mathbb{Z}^s,\; \sum|k_i| \leq K}\left(\prod^s_{j=1}\cos \left(\pi k_j \frac{x_j - a_j}{b_j - a_j}\right)\right) e^{\imath \pi {\bf k}\cdot{\bf \frac{y-a}{b-a}}}. \label{equation_truncated_kernel}
\end{align}
For any expectation that satisfies the conditions in the previous section, we have our full approximant: 
\begin{align}
	&
	\mathbb{E}[\mathfrak{f}({\bf Y})]\nonumber \\
	\approx & 
	\int_{\mathbb{R}^s} 
	\frac{1}{N} \sum^{N-1}_{n=0} 
	\bar{\mathfrak{f}}\left({\bf p}_{\boldsymbol{\phi}, n}\right) 
	\mathcal{K}^{\cos}_{s, K}\left({\bf p}_{\boldsymbol{\phi}, n}, {\bf y}\right) \mu (d{\bf y})\nonumber\\
	= &
	\frac{1}{N} \sum^{N-1}_{n=0} 
	\bar{\mathfrak{f}}\left({\bf p}_{\boldsymbol{\phi}, n}\right)
	\int_{\mathbb{R}^s} 
	\sum_{{\bf k} \in \mathbb{Z}^s,\; \sum|k_i| \leq K}\left(\prod^s_{j=1}\cos \left(\pi k_j \phi\left(\left\{\frac{n {g_j}}{N}\right\}\right)\right)\right) e^{\imath \pi {\bf k}\cdot{\bf\frac{y-a}{b-a}}}
	\mu (d{\bf y}) \nonumber\\
	= &
	\frac{1}{N} \sum^{N-1}_{n=0} 
	\bar{\mathfrak{f}}\left({\bf p}_{\boldsymbol{\phi}, n}\right)
	\left(
		\sum_{{\bf k} \in \mathbb{Z}^s,\; \sum|k_i| \leq K}\left(\prod^s_{j=1}\cos \left(\pi k_j \phi\left(\left\{\frac{n {g_j}}{N}\right\}\right)\right)\right) 
		e^{-\imath \pi {\bf k}\cdot{\bf\frac{a}{b-a}}}
		\int_{\mathbb{R}^s} 
		e^{\imath \pi {\bf k}\cdot{\bf\frac{y}{b-a}}}
		\mu (d{\bf y})
	\right) \nonumber\\
	= &
	\frac{1}{N} \sum^{N-1}_{n=0} 
    \bar{\mathfrak{f}}\left({\bf p}_{\boldsymbol{\phi}, n}\right)
    \left(
    \sum_{{\bf k} \in \mathbb{Z}^s,\; \sum|k_i| \leq K}\left(\prod^s_{j=1}\cos \left(\pi k_j \phi\left(\left\{\frac{n {g_j}}{N}\right\}\right)\right)\right) 
    e^{-\imath \pi {\bf k}\cdot{\bf\frac{a}{b-a}}}
    \mathcal{F}_{\mu}\left({\frac{{\bf k}\pi }{\bf b-a}}\right)
	\right), \label{equation_cosine_full_scheme}
\end{align}
where $\mathcal{F}_\mu$ is the Fourier transform of the measure $\mu$.
Recall from Equation \eqref{equation_reproducing_property} that $\mathfrak{f}$ and $\bar{\mathfrak{f}}$ coincide in $\mathcal{D}$.
We can simply replace $\bar{\mathfrak{f}}$ by $\mathfrak{f}$ in the above scheme.

In practice, one would first calculate the expected reproducing kernel value $\int_{\mathbb{R}^s} \mathcal{K}^{\cos}_{s, K}\left({\bf p}_{\boldsymbol{\phi}, n}, {\bf y}\right) \mu (d{\bf y})$ for each lattice point, possibly in an offline setting and then calculate the main sum. 

Define the absolute approximation error under this setting as $\epsilon$ and we see that 
\begin{align*}
	\epsilon 
	:= & 
	\left| \mathbb{E}[\mathfrak{f}({\bf Y})] -
	\int_{\mathbb{R}^s} 
	\frac{1}{N} \sum^{N-1}_{n=0} 
	\mathfrak{f}\left({\bf p}_{\boldsymbol{\phi}, n}\right) 
	\mathcal{K}^{\cos}_{s, K}\left({\bf p}_{\boldsymbol{\phi}, n}, {\bf y}\right) \mu (d{\bf y})\right| \\
	\leq & 
	\left| \mathbb{E}[\mathfrak{f}({\bf Y})] - \mathbb{E}[\bar{\mathfrak{f}}({\bf Y})]\right| \\
	&
	+ \left| 
		\mathbb{E}[\bar{\mathfrak{f}}({\bf Y})] 
		- \int_{\mathbb{R}^s} 
		\frac{1}{N} \sum^{N-1}_{n=0} 
		\bar{\mathfrak{f}}\left({\bf p}_{\boldsymbol{\phi}, n}\right) 
		\mathcal{K}^{\cos}_s\left({\bf p}_{\boldsymbol{\phi}, n}, {\bf y}\right) \mu (d{\bf y})
	\right| \\
	&
	+ \left| 
		\int_{\mathbb{R}^s} 
		\frac{1}{N} \sum^{N-1}_{n=0} 
		\bar{\mathfrak{f}}\left({\bf p}_{\boldsymbol{\phi}, n}\right) 
		\mathcal{K}^{\cos}_s\left({\bf p}_{\boldsymbol{\phi}, n}, {\bf y}\right) \mu (d{\bf y})
		-
		\int_{\mathbb{R}^s} 
		\frac{1}{N} \sum^{N-1}_{n=0} 
		\bar{\mathfrak{f}}\left({\bf p}_{\boldsymbol{\phi}, n}\right) 
		\mathcal{K}^{\cos}_{s, K}\left({\bf p}_{\boldsymbol{\phi}, n}, {\bf y}\right) \mu (d{\bf y}) 
	\right|\\
	\leq & \mathbb{E}[\left| \mathfrak{f}({\bf Y}) - \bar{\mathfrak{f}}({\bf Y})\right| {\bf 1}_{\{\mathbb{R}^s\backslash \mathcal{D}\}}({\bf Y})] 
	+ \left(\sum_{{\bf h}\in L^\bot\backslash\{{\bf 0}\}}r_{\alpha,\boldsymbol{\gamma}, s}({\bf h})\right)^\frac{1}{2}\left(\prod^s_{i=1}C_i\right)^{\frac{1}{2}}
	||\mathfrak{f}||_{\mathcal{K}^{\cos}_{\alpha, \boldsymbol{\gamma}, s}(\mathcal{D})} \\
	&
	+ \frac{1}{N} \sum^{N-1}_{n=0} 
	\left|\bar{\mathfrak{f}}\left({\bf p}_{\boldsymbol{\phi}, n}\right)\right| \left(
	\sum_{{\bf k} \in \mathbb{Z}^s,\; \sum|k_i| > K}	
	\left|\int_{\mathbb{R}^s} 
	\left(\prod^s_{j=1}\cos \left(\pi k_j \frac{y_j - a_j}{b_j -a_j}\right) \right)
	\mu (d{\bf y})\right|\right).
\end{align*}

For the third term, a bit more detail is required.
Again we have the convergence assumption on the cosine transform of measure $\mu$ with the additional condition that $\beta > s$.
With the inequality $(K -1)\times 1 < (K-2) \times 2 < \ldots < \left(\frac{K}{2}\right)^2$, for all $K$ with $K \in \mathbb{N}$ and $K > 2$ and mathematical induction, one can show that if $\sum|k_i| = K$, $$(r_{\beta, \boldsymbol{\rho}, s}({\bf k}))^\frac{1}{2} \leq |K - s|^{-\beta}\prod^s_{j=1}\max\{1, \sqrt{\rho_j}\}.$$
There are in total $\frac{1}{(s-1)!}\prod^{s-1}_{j =1}(K + j)$ s-tuples of natural numbers satisfying $\sum^s_{i =1}k_i = K$ for $s\geq 2$. (see \cite{lisi_2007_some} for further explanation.)
Taking into account all the possible combinations of positive and negative signs among all dimensions when we consider all integers, we have 
\begin{align*}
	\sum_{{\bf k} \in \mathbb{Z}^s,\; \sum|k_i| > K}(r_{\beta, \boldsymbol{\rho}, s}({\bf k}))^\frac{1}{2} 
	\leq & 
	\sum^\infty_{k = K+1}2^s \frac{1}{(s-1)!}\prod^{s-1}_{j =1}(k + j) \times |k - s|^{-\beta}\prod^s_{j=1}\max\{1, \sqrt{\rho_j}\}\\
	\leq &
	\frac{2^s}{(s-1)!} \left(1 + \frac{2s-1}{K}\right)^{s-1}\prod^s_{j=1}\max\{1, \sqrt{\rho_j}\}
	\sum^\infty_{k = K+1} |k - s|^{-(\beta + 1 -s)}.
\end{align*}
The final inequality also holds when $s=1$.

Therefore, we find the following error bound for the lattice expansion scheme:
\begin{align}
\epsilon
\leq & \mathbb{E}[\left| \mathfrak{f}({\bf Y}) - \bar{\mathfrak{f}}({\bf Y})\right| {\bf 1}_{\{\mathbb{R}^s\backslash \mathcal{D}\}}({\bf Y})] 
+ \left(\sum_{{\bf h}\in L^\bot\backslash\{{\bf 0}\}}r_{\alpha,\boldsymbol{\gamma}, s}({\bf h})\right)^\frac{1}{2}\left(\prod^s_{i=1}C_i\right)^{\frac{1}{2}}
||\mathfrak{f}||_{\mathcal{K}^{\cos}_{\alpha, \boldsymbol{\gamma}, s}(\mathcal{D})} \nonumber\\
&
+ \frac{1}{N} \sum^{N-1}_{n=0} 
\left|\bar{\mathfrak{f}}\left({\bf p}_{\boldsymbol{\phi}, n}\right)\right| 
\frac{2^s}{(s-1)!} \left(1 + \frac{2s-1}{K}\right)^{s-1}\prod^s_{j=1}\max\{1, \sqrt{\rho_j}\}
|K - s|^{-(\beta -s)}. \label{equation_total_error_bound}
\end{align}

To conclude, we can apply the telescoping technique and describe our error by three separate pieces: the projection error from a non-periodic to a periodic function, the lattice integration error and the kernel truncation error. 
With this proof, we successfully extended the lattice integration in \cite{dick_2014_lattice} to a more general setting and made use of both the smoothness of the integrand in terms of cosine coefficients and the convergence of the cosine transform with respect to the probability measure.
However, with our ``number theory based'' derivation, the model dimension $s$ still heavily impacts the complete error bound, which may be an obstacle when applying the method to higher-dimensional problems.
Moreover, the assumption of the cosine transform's decay is rough.
Results from Fourier analysis may be incorporated to further improve the error bound.
Further study for the error control is therefore recommended.

\section{Discussion}\label{section_discussion}

In Section \ref{section_scheme}, we have provided a justification for the cosine expansion lattice scheme. 
We extended the tent transformed lattice rule in the half-cosine space to general probability measures, with its error controlled by the decay rate of the cosine transform of the probability measures as well as the cosine coefficients of the integrands.

However, there are remaining questions. Is there a better way to control the approximation error than simple algebraic manipulation? How does our current scheme connect to previous work on wavelets? 
Here we aim to present some discussion on the above issues. 

\subsection{Alternative Error Formulation}

The first possibility we would like to consider is whether the error of the approximation in Section \ref{section_lattice_integration} can be written in an alternative form. 
By doing so, we aim to avoid bounding the term $I({\bf h})$ along each dimension, as this estimation is rough.  

We revisit the derivation of $\epsilon_2$ in Section \ref{section_lattice_integration}, specifically the right-hand side of Equation \eqref{equation_error_expression}. 
Note that the innermost sum in the expression can be rewritten as follows:
\begin{align*}
&
\sum_{{\bf k}\in\mathbb{Z}^s}
(\sqrt{2})^{-|{\bf h}-{\bf k}|_0}\widetilde{\mathfrak{f}}_{\cos}({\bf h}-{\bf k})
\left(\prod^s_{j=1}\cos\left(\pi k_j \frac{y_j-a_j}{b_j-a_j}\right)\right)\\
= &
\sum_{\vec{k}\in\mathbb{Z}^s} 
\left(\prod^s_{j=1}\cos\left(\pi k_j \frac{y_j-a_j}{b_j-a_j}\right)\right)
\int_{[0,1]^s}\mathfrak{f}({\bf z}\times {\bf (b-a)} + {\bf a})\prod^s_{j=1}\cos(\pi (h_j-k_j)z_j)d{\bf z}\\
=&
\sum^\infty_{k_1 = 1}\sum_{{\bf k}_{-1}\in\mathbb{Z}^{s-1}} 
\cos\left(\pi k_1 \frac{y_1-a_1}{b_1-a_1}\right)\left(\prod^s_{j=2}\cos\left(\pi k_j \frac{y_j-a_j}{b_j-a_j}\right)\right)\times \\
&
\int_{[0,1]^s}\mathfrak{f}({\bf z}\times {\bf (b-a)} + {\bf a})\cos(\pi (h_1-k_1)z_1)\prod^s_{j=2}\cos(\pi (h_j-k_j)z_j)d{\bf z}\\
&
+\sum_{{\bf k}_{-1}\in\mathbb{Z}^{s-1}} 
\left(\prod^s_{j=2}\cos\left(\pi k_j \frac{y_j-a_j}{b_j-a_j}\right)\right)
\int_{[0,1]^s}\mathfrak{f}({\bf z}\times {\bf (b-a)} + {\bf a})\cos(\pi h_1z_1)\prod^s_{j=2}\cos(\pi (h_j-k_j)z_j)d{\bf z}\\
&
+ \sum^{-1}_{k_1 = -\infty}\sum_{{\bf k}_{-1}\in\mathbb{Z}^{s-1}} 
\cos\left(\pi k_1 \frac{y_1-a_1}{b_1-a_1}\right)\left(\prod^s_{j=2}\cos\left(\pi k_j \frac{y_j-a_j}{b_j-a_j}\right)\right)\times \\
&
\int_{[0,1]^s}\mathfrak{f}({\bf z}\times {\bf (b-a)} + {\bf a})\cos(\pi (h_1-k_1)z_1)\prod^s_{j=2}\cos(\pi (h_j-k_j)z_j)d{\bf z},
\end{align*}
by using the definition of cosine coefficients.
We separate the sum here in three parts according to whenever $k_1$ is positive, negative or zero. 
Next, we combine the terms whose $|k_1|$ value is the same. 
\begin{align*}
&
\sum_{{\bf k}\in\mathbb{Z}^s}
(\sqrt{2})^{-|{\bf h}-{\bf k}|_0}\widetilde{\mathfrak{f}}_{\cos}({\bf h}-{\bf k})
\left(\prod^s_{j=1}\cos\left(\pi k_j \frac{y_j-a_j}{b_j-a_j}\right)\right)\\
= & 
\sum_{{\bf k}_{-1}\in\mathbb{Z}^{s-1}} 
\left(\prod^s_{j=2}\cos\left(\pi k_j \frac{y_j-a_j}{b_j-a_j}\right)\right)
\int_{[0,1]^s}\mathfrak{f}({\bf z}\times {\bf (b-a)} + {\bf a})\cos(\pi h_1z_1)\prod^s_{j=2}\cos(\pi (h_j-k_j)z_j)d{\bf z}\\
&
+ \sum^\infty_{k_1 = 1}\sum_{{\bf k}_{-1}\in\mathbb{Z}^{s-1}} 
\cos\left(\pi k_1 \frac{y_1-a_1}{b_1-a_1}\right)\left(\prod^s_{j=2}\cos\left(\pi k_j \frac{y_j-a_j}{b_j-a_j}\right)\right)\times \\
&
\int_{[0,1]^s}\mathfrak{f}({\bf z}\times {\bf (b-a)} + {\bf a})(\cos(\pi (h_1+k_1)z_1)+ \cos(\pi (h_1-k_1)z_1)) \prod^s_{j=2}\cos(\pi (h_j-k_j)z_j)d{\bf z}\\
= &
\sum_{k_1 \in \mathbb{N}_0}\sum_{{\bf k}_{-1}\in\mathbb{Z}^{s-1}} 
\left(\prod^s_{j=1}\cos\left(\pi k_j \frac{y_j-a_j}{b_j-a_j}\right)\right)\times \\
&
2^{|k_1|_0} \int_{[0,1]^s}\mathfrak{f}({\bf z}\times {\bf (b-a)} + {\bf a})\cos(\pi h_1 z_1) \cos(\pi k_1 z_1) \prod^s_{j=2}\cos(\pi (h_j-k_j)z_j)d{\bf z}.
\end{align*}

Then we repeat the similar steps of separating terms and combining those with the same absolute value  for $k_2, \cdots, k_s$, which leads to:
\begin{align*}
&
\sum_{{\bf k}\in\mathbb{Z}^s}
(\sqrt{2})^{-|{\bf h}-{\bf k}|_0}\widetilde{\mathfrak{f}}_{\cos}({\bf h}-{\bf k})
\left(\prod^s_{j=1}\cos\left(\pi k_j \frac{y_j-a_j}{b_j-a_j}\right)\right)\\
= & 
\sum_{{\bf k}\in \mathbb{N}^s_0} 
\int_{[0, 1]^s}\mathfrak{f}({\bf z}\times {\bf (b-a)} + {\bf a})
\left(\prod^s_{j=1}\cos(\pi h_j z_j)\right)2^{|{\bf k}|_0/2}\prod^s_{j=1}\cos(\pi k_j z_j)d{\bf z}
2^{|{\bf k}|_0/2}\left(\prod^s_{j=2}\cos\left(\pi k_j \frac{y_j-a_j}{b_j-a_j}\right)\right).
\end{align*}
This can be seen as the half-period cosine expansion of the function $$\mathfrak{f}_{{\bf h}}({\bf y}) :=\mathfrak{f}({\bf y})\left(\prod_{j=1}^{s}\cos \left(\pi h_j \frac{y_j-a_j}{b_j-a_j} \right)\right).$$

Therefore, the right-hand side of Equation \eqref{equation_error_expression} can be rewritten as 
\begin{align*}
&
\int_{\mathbb{R}^s}
\sum_{{\bf h}\in L^\bot\backslash\{{\bf 0}\}}\sum_{{\bf k}\in \mathbb{Z}^s}
(\sqrt{2})^{-|{\bf h}-{\bf k}|_0}\widetilde{f}_{\cos}({\bf h}-{\bf k})
\left(\prod^s_{j=1}\cos \left(\pi k_j \frac{y_j - a_j}{b_j -a_j}\right) \right)
\mu(d{\bf y})\\
=& 
\sum_{{\bf h}\in L^\bot\backslash\{{\bf 0}\}}
\int_{\mathbb{R}^s}
\bar{\mathfrak{f}}_{{\bf h}}({\bf y})
\mu(d{\bf y}).
\end{align*}

To summarize, the error of this approximation is just the sum of functions $\bar{\mathfrak{f}}_{{\bf h}}$ integrated over the probability measure $\mu$ over the dual lattice.
The approximation error is controlled by a series with individual term  $\int_{\mathbb{R}^s}
\bar{\mathfrak{f}}_{{\bf h}}({\bf y})
\mu(d{\bf y})$, a cosine transform with respect to the cosine series of $\mathfrak{f}$ under the probability measure $\mu$. 
The decay rate of such integral, depending on the combined smoothness of the measure $\mu$ and the function $\mathfrak{f}$, is the key to limit the error of our scheme.

\subsection{Cosine Wavelets} \label{section_cosine_wavelets}

We place our work in the context of the construction of wavelets in \cite{chau_2018_wavelet} to see the similarities and differences between the two approaches. 
Instead of using the full Fourier series for the wavelet construction, as in \cite{chau_2018_wavelet}, we choose a ``cosine function only" basis set as our starting point, corresponding to the half-period cosine space we used in Section \ref{section_scheme}. 

Under this setting, we start our derivation from the set $$\Gamma_{a,b}:=\left\{\frac{1}{\sqrt{2}},\left.\cos\left(k\pi\frac{y-a}{b-a}\right)\right| k=1,2,\ldots\right\},$$ with $[a,b]\in\mathbb{R}$ a finite range. 
We define an inner product
\begin{equation*}
<\mathfrak{f},\mathfrak{g}>_{L^2([a,b])}:=\frac{2}{b-a}\int^{b}_{a}\mathfrak{f}(y)\mathfrak{g}(y) dy,
\end{equation*}  
and the corresponding norm $||\cdot||_{L^2([a, b])}$.
We denote any function $\mathfrak{f}$ such that $||\mathfrak{f}||_{L^2([a, b])} < \infty$, as $\mathfrak{f} \in L^2([a, b])$.

Equipped with the above definitions, we construct an approximation space together with a localized basis. Consider the following function $\mathcal{K}^{wl}_{N'}:\mathbb{R} \times [N'] \rightarrow\mathbb{R}$,
\begin{align}
\mathcal{K}^{wl}_{N'} (x, r)
:=& 
\frac{1}{2}
+\sum^{N'-1}_{k=1}\cos\left(k\pi\frac{x-a}{b-a}\right)\cos\left(k\pi\frac{2r-1}{2N'}\right)
\nonumber\\
=& 
\frac{1}{2}
+\frac{1}{2}\sum^{N'-1}_{k=1}\cos\left(
k\pi
\left(\frac{x-a}{b-a}-\frac{2r-1}{2N'}\right)
\right)
+\frac{1}{2}\sum^{N'-1}_{k=1}\cos\left(
k\pi
\left(\frac{x-a}{b-a}+\frac{2r-1}{2N'}\right)
\right)\label{kernel}\\
=& \left\{
\begin{array}{ll}
\frac{N'}{2}&\mbox{ if } \frac{x-a}{b-a} = 2l \pm \frac{2r-1}{2N'} \mbox{ for } l \mbox{ an integer,}\\
\frac{\sin((N'-\frac{1}{2})(\frac{x-a}{b-a}-\frac{2r-1}{2N'})\pi)}{4\sin(\frac{\pi}{2}(\frac{x-a}{b-a}-\frac{2r-1}{2N'}))}
+\frac{\sin((N'-\frac{1}{2})(\frac{x-a}{b-a}+\frac{2r-1}{2N'})\pi)}{4\sin(\frac{\pi}{2}(\frac{x-a}{b-a}+\frac{2r-1}{2N'}))}	& \mbox{ otherwise,}
\end{array}
\right.\nonumber
\end{align}
where $r=1, 2, \ldots, N'$. This definition is a special case of the scaling functions given in Equation (2.13) of \cite{fischer_1997_wavelets}, in which the authors presented an uniform approach for the construction of wavelets based on orthogonal polynomials. 
The properties of $\mathcal{K}^{wl}_{N'}$, that are relevant to our numerical method are listed in the next proposition.
\begin{proposition}
	The function $\mathcal{K}^{wl}_{N'}$, which is defined in Equation (\ref{kernel}), satisfies the following properties:
	\begin{enumerate}
		\item[(a)] The inner product of two scaling functions is given by the following equation:
		\begin{equation*}
		<\mathcal{K}^{wl}_{N'}(\cdot, r),\mathcal{K}^{wl}_{N'}(\cdot, q)> = \mathcal{K}^{wl}_{N'}\left(a+\frac{2r-1}{2N'}(b-a), q\right), \qquad r,q = 1, 2, \ldots, N'.
		\end{equation*}
		Thus, $\{\mathcal{K}^{wl}_{N'}(x, r)|r = 1, \ldots, N'\}$ is an orthogonal set.
		
		\item[(b)] The scaling function $\mathcal{K}^{wl}_{N'}(\cdot, r)$ is localized around $a+\frac{2r-1}{2N'}(b-a)$. By this we mean that for the subspace
		$$V_{N'}:=\mbox{span}\left\{
		\frac{1}{\sqrt{2}},
		\left.
		\cos\left(k\pi\frac{y-a}{b-a}\right)
		\right|
		k = 1, 2, \ldots, N'-1
		\right\},$$
		we have 
		$$\left|\left|
		\frac{\mathcal{K}^{wl}_{N'}(\cdot, r)}{\mathcal{K}^{wl}_{N'}(a+\frac{2r-1}{2N'}(b-a), r)}
		\right|\right|_{L^2([a, b])}
		=\min\left\{||\mathfrak{f}||_{L^2([a,b])}:\mathfrak{f}\in V_{N'}, \mathfrak{f}\left(a+\frac{2r-1}{2N'}(b-a)\right)=1\right\} .$$
		
		\item[(c)] $\{\mathcal{K}^{wl}_{N'}(\cdot, r)|r=1, 2,\ldots,N'\}$  is a basis for $V_{N'}$.
		
		\item[(d)] The scaling function $\mathcal{K}^{wl}_{N'}$ is also a kernel polynomial in the sense that, for any function $v$ in $V_{N'}$, we have $$<v,\mathcal{K}^{wl}_{N'}(\cdot, r)>_{L^2([a,b])} = v\left(a+\frac{2r-1}{2J}(b-a)\right).$$
	\end{enumerate}
\end{proposition}
Readers are referred to \cite{fischer_1997_wavelets} for further properties of such functions and for a condensed proof of the above properties one may follow the derivation of Theorem 3.1 in \cite{chau_2018_wavelet}.

Applying the results obtained above, we may define the wavelet expansion $\mathfrak{f}_{wl}: \mathbb{R} \rightarrow \mathbb{R}$ for any function $\mathfrak{f} \in L^2([a, b])$ by $$\mathfrak{f}_{wl}(y) := \frac{2}{N'}\sum^{N'}_{r=1}\mathfrak{f}\left(a+\frac{2r-1}{2N'}(b-a)\right)\mathcal{K}^{wl}_{N'}(y, r),$$ and approximate the expectation $\mathbb{E}[\mathfrak{f}(Y)]$ by,
\begin{align}
	\mathbb{E}[\mathfrak{f}(Y)]  
	\approx & 
	\mathbb{E}\left[
		\frac{2}{N'}
		\sum^{N'}_{r=1}\mathfrak{f}\left(a+\frac{2r-1}{2N'}(b-a)\right)\mathcal{K}^{wl}_{N'}(Y, r)
	\right] \nonumber\\
	= &
	\frac{2}{N}\sum^{N'}_{r=1}\mathfrak{f}\left(a+\frac{2r-1}{2J}(b-a)\right)
	\mathbb{E}\left[\frac{1}{2}
	+\sum^{N'-1}_{k=1}\cos\left(k\pi\frac{Y-a}{b-a}\right)\cos\left(k\pi\frac{2r-1}{2N'}\right)\right]\nonumber\\
	= &
	\frac{1}{N'}\sum^{N'}_{r=1}\mathfrak{f}\left(a+\frac{2r-1}{2J}(b-a)\right)
	\mathbb{E}\left[1
	+\sum^{N'-1}_{k=1}(e^{\imath k\pi\frac{Y-a}{b-a}} + e^{-\imath k\pi\frac{Y-a}{b-a}})\cos\left(k\pi\frac{2r-1}{2N'}\right)\right]\nonumber\\
	= &
	\frac{1}{N'}\sum^{N'}_{r=1}\mathfrak{f}\left(a+\frac{2r-1}{2J}(b-a)\right)
	\sum^{N'-1}_{k=1-N'}\cos\left(k\pi\frac{2r-1}{2N'}\right)e^{-\imath \pi k\frac{a}{b-a}} \mathbb{E}\left[e^{\imath \frac{k\pi}{b-a} Y} \right].\nonumber
\end{align}

Comparing this expression with Equation \eqref{equation_cosine_full_scheme} when $s = 1$, it is clear that these two formulas are of the same form. 
The main difference is that the wavelet formula can be seen as an expansion scheme using the lattice points 
$$P_{wl}(N', [a, b]) := \left\{a + n\frac{2r-1}{2N'}(b-a): 1 \leq r \leq N'\right\},$$
with the kernel also bounded at the corresponding value $N'-1$. 

Considering a two-dimensional function $\mathfrak{f}: \mathbb{R}^2 \rightarrow \mathbb{R}$, we can extend the cosine wavelet to two dimensions by applying the expansion to each dimension separately. 
\begin{align*}
&
\mathbb{E}[\mathfrak{f}(Y_1, Y_2)]\\
= &
\mathbb{E}\left[
\frac{2}{N'}\sum^{N'}_{r_1 = 1}\mathfrak{f}\left(a_1 + \frac{2r_1 -1}{2N'}(b_1-a_1), Y_2\right) \mathcal{K}^{wl}_{N'}(Y_1, r_1)
\right]\\
= & 
\mathbb{E}\left[
\frac{4}{{N'}^2}\sum^{N'}_{r_1 = 1}\sum^{N'}_{r_2 = 1}\mathfrak{f}\left(a_1 + \frac{2r_1 -1}{2N'}(b_1-a_1), a_2 + \frac{2r_2-1}{2N'}(b_2-a_2)\right) \mathcal{K}^{wl}_{N'}(Y_1, r_1)\mathcal{K}^{wl}_{N'}(Y_2, r_2)
\right]\\
= &
\frac{1}{{N'}^2}\sum^{N'}_{r_1 = 1}\sum^{N'}_{r_2 = 1}\mathfrak{f}\left(a_1 + \frac{2r_1 -1}{2N'}(b_1-a_1), a_2 + \frac{2r_2-1}{2N'}(b_2-a_2)\right) \times \\
&
\sum^{N'-1}_{k_1 = 1-N'}\sum^{N'-1}_{k_2 = 1-N'}\left(\prod_{j = 1}^s\cos \left(k_j \pi\frac{2r_j-1}{2N'}\right)\right) e^{-\imath \pi {\bf k}\cdot {\bf \frac{a}{b-a}}}\mathbb{E}\left[	e^{\imath \frac{{\bf k}\pi}{\bf b-a}\cdot {\bf Y}}\right].
\end{align*} 

Again this is in a similar form as in Equation \eqref{equation_cosine_full_scheme}. 
However, the total number of terms ${N'}^2$ increases with the dimension and there is no clear way to select the more significant ones. 

One of the key advantages of the cosine expansion lattice scheme is that we have removed the link between the quadrature points and the number of summation terms $N$. 
This allows us to apply an online/offline construction in our algorithm.
We can first compute the expectations of reproducing kernels $\mathbb{E}\left[
\mathcal{K}^{\cos}_{s, K}\left({\bf p}_{\boldsymbol{\phi}, n}, {\bf Y}\right)\right]$ for all points in the lattice sequence and use the stored values in the approximation scheme. 
This construction is not feasible in the cosine wavelets setting and it is one of the advancements from the cosine expansion lattice scheme.  

\section{Numerical Experiments} \label{section_numerical_experiment}

In this section we perform numerical experiments to test the cosine expansion lattice scheme. 

In order to demonstrate that our scheme can inherit the results from the previous literature, we use the lattice sequence from \cite{hickernell_2012_weighted} and perform the tent-transformation on them in all our tests. 
This sequence has been used in the numerical results section of \cite{dick_2014_lattice} and readers are referred to the references therein for further information. 

All figures displaying numerical results in this Section present the log absolute error $\log_{10}(|\epsilon|)$ against the log number of lattice points $\log_{10}(N)$.

For the first two experiments, we compute the expectation with the following test function: $$\mathfrak{f}^1_{s, w}({\bf y}):= \prod_{j=1}^{s}\left(1+\frac{w^j}{21}(-10+42y^2_j - 42y^5_j +21 y^6_j)\right).$$

\subsection{Uniform Distribution}

Whenever the reproducing kernel's expectation $\mathbb{E}[\mathcal{K}^{\cos}_{s} \left( {\bf p}_{\boldsymbol{\phi}, n}, {\bf Y }\right)]$ is known explicitly, we can simplify the algorithm to its original form in Equation \eqref{equation_quadrature_rule} instead of the full scheme in Equation \eqref{equation_cosine_full_scheme}. 
This results in an algorithm that has no difference to the original QMC rule in terms of computational complexity.
As stated in Remark \ref{remark_uniform_distribution}, this is indeed the case for the uniform distribution.

In this subsection, we approximate $$\mathbb{E}[\mathfrak{f}^1_{s, w}({\bf Y})],$$ where ${\bf Y}$ follows an uniform distribution on the domain $[0, 1]\times [-1, 1] \times [0, 1] \times [-1, 1]\times \cdots$ for dimension $s \in [8]$.
The reference value is $\prod_{j=1}^s\left(2 + w^j \frac{2}{3}\right)^{(j\mod 2)}$. 

Note that we are not establishing a new numerical result or proof of convergence here as this problem can be solved by the original QMC algorithm through a change of variables. 
These tests serve two purposes. 
They establish the base line result as a comparison to the results in the next section and allow us to provide comments on the actual implementation of the QMC algorithm. 

The result for $w = 0.5$ can be seen in Figure \ref{figure_uniform}.
All tests use $2^{20}$ evaluations for their final result.  
The reader should note that instead of having 1 as the common reference value throughout all dimensions, as in \cite{dick_2014_lattice}, the reference values of our setting follow those in Table \ref{table_reference_uniform}. 
This is by design to demonstrate some key properties of lattice expansions. 
Under this setting, results from neighboring dimensions form a pair, for example, dimensions 4 and 5 share the same reference value. 
As demonstrated in Figure \ref{figure_uniform}, the absolute errors converge to similar limits for these pairs, while the results worsen when the reference increases.
This demonstrates that the QMC rules can be generalized to higher dimensions, but their error depends on the value of the target integration, which fits the result from Theorem \ref{theorem_lattice}. 

Finally, it should be noted that the result for dimension 1 is particular strong. 
This is because the function evaluation for $y_j \in [0,1)$, the monomial $y^k_j$ would always be around zero which helps stabilizing the function evaluation.
When we use a projection domain that is larger than the unit box when approximating a polynomial, as we do in the upcoming examples, the results may seem to be worse than standard QMC result and this is one of the contributing factors.  

\begin{table}[t!]
	\centering
	\begin{tabular}{|c|c|c|c|c|c|c|c|c|}
		\hline
		Dimension & 1 & 2 & 3 & 4 & 5 & 6 & 7 & 8 \\
		\hline
		Ref. Values & 1 & 2.167 & 2.167 & 4.424 & 4.424 & 8.893 & 8.893 & 17.81\\
		\hline
	\end{tabular}
	\caption{Reference solution for Uniform Distribution} \label{table_reference_uniform}
\end{table}

\begin{figure}[t!]
	\includegraphics{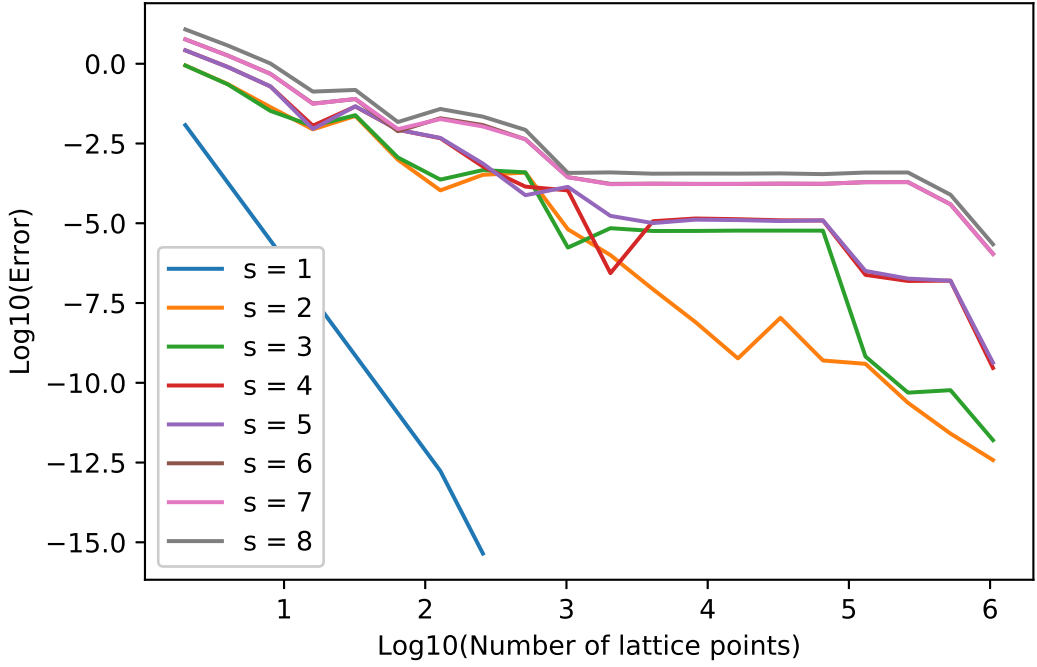}
	\caption{Absolute error verses the number of lattice points for uniform distribution. }\label{figure_uniform}
\end{figure}

\subsection{Normal Distribution}

In this subsection, we approximate $$\mathbb{E}[\mathfrak{f}^1_{s, 0.9}({\bf Y})],$$ where ${\bf Y}$ follows a multivariate normal distribution with mean $0$ for all dimensions and covariance matrix $$\left(\begin{array}{ccc}
0.5^2 & &\\
& \ddots & \\
& & 0.5^2
\end{array}\right).$$ 
Further, we restrict our lattice to the box $[-4.5, 4.5]^s$ and we study the cases of dimension $s \in [3]$.
We applied the full scheme in Equation \eqref{equation_cosine_full_scheme} here with $K = 2^7$.
The reference values listed in vector form are $(1.2324, 1.4901, 1.7705)$.

The convergence result can be seen in Figure \ref{figure_normal}.
All tests use $2^{18}$ evaluations for their final result.  
The approximation value converges to the references value in all tests, as we would expect.
We can observe the structure of the error plateauing for a while with respect to the number of lattice points and suddenly improving when a critical number is reached in Figure \ref{figure_normal}. 
This seems to be related to the lattice sequence.

\begin{figure}[t!]
	\includegraphics{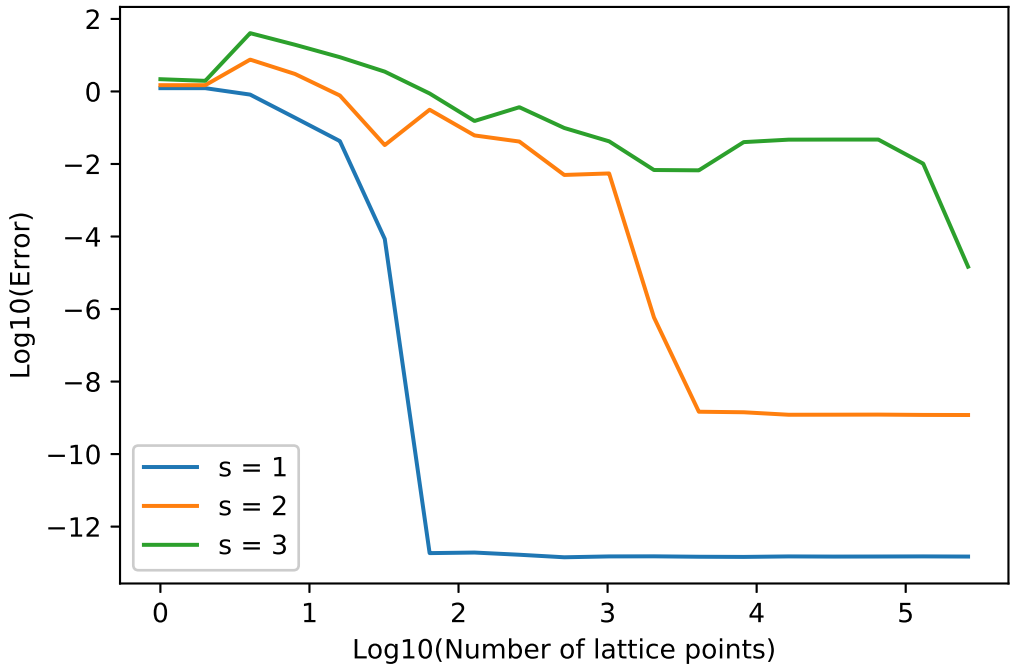}
	\caption{Absolute error verses the number of lattice points for normal distribution. }\label{figure_normal}
\end{figure}

In the case of the normal distribution, there is a final level of error for each dimension, where increasing the number of lattice points further would not improve the approximation result. 
This is related to the error of projecting a non-periodic function to the half-period cosine space and the truncation of reproducing kernel. 
These errors dominate when the lattice rule error has converged and limit the final result.

\subsubsection{Further studies on parameters}

With the successful implementation of our scheme, we conducted further tests for the selection of the kernel truncation parameter and projection domain.  
We did not consider the generation and construction of different lattice sequences in this article considering the large amount of previous research and work on this topic.
It is difficult to draw a satisfying picture on this topic from a simple test.
However, we would further consider the choices of projection domain and kernel truncation limit $K$ and their impact on the approximation error. 
In the following tests, we focus on the 2 dimensional case.

\begin{figure}[t!]
	\includegraphics{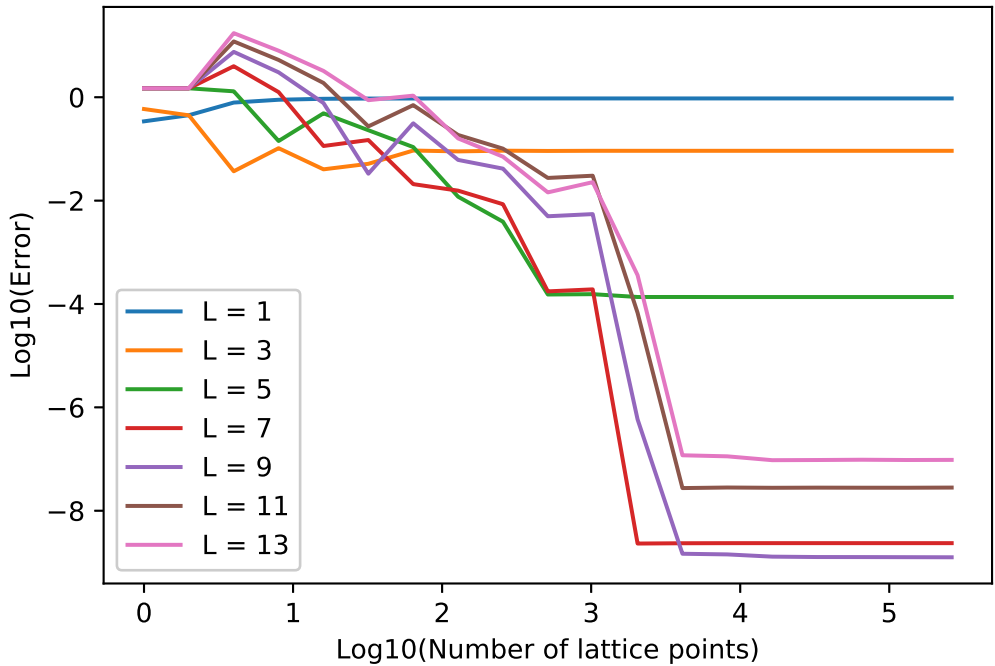}
	\caption{Absolute error verses the number of lattice points for the normal distribution for different projection domains.}\label{figure_normal_boundary}
\end{figure}

In Figure \ref{figure_normal_boundary}, we see the result of the same test we used for the normal distribution but instead of changing the test dimension $s$, we fixed the test dimension to 2 and varied the projection domain through a variable $L$. 
The projection domain is defined as $[-0.5L, 0.5L]^2$ and the kernel truncation $K$ is fixed at $2^7$. 

From the result, we can see that the convergence behavior is similar among all tests for $L \geq 5$.
Our result implies that the lattice sequence has the most significant influence on the convergence behavior with respect to the number of lattice points of our scheme. 
When most of the probability mass is included in the projection range ($L \geq 5$ here), the projection range only affects the final error. 
As one can see, the error is the lowest when $L = 9$, which can possibly be explained by the expression in Equation \eqref{equation_total_error_bound}. 
The first term at the right hand side decreases when the probability mass outside of the projection range increases, while the rest of the terms relate to the averages of the integrand and increases with the projection range.
We have to find the balance for these two competing effects when we choose the projection range.
Detailed results can be seen in the appendix.

\begin{figure}[t!]
	\includegraphics{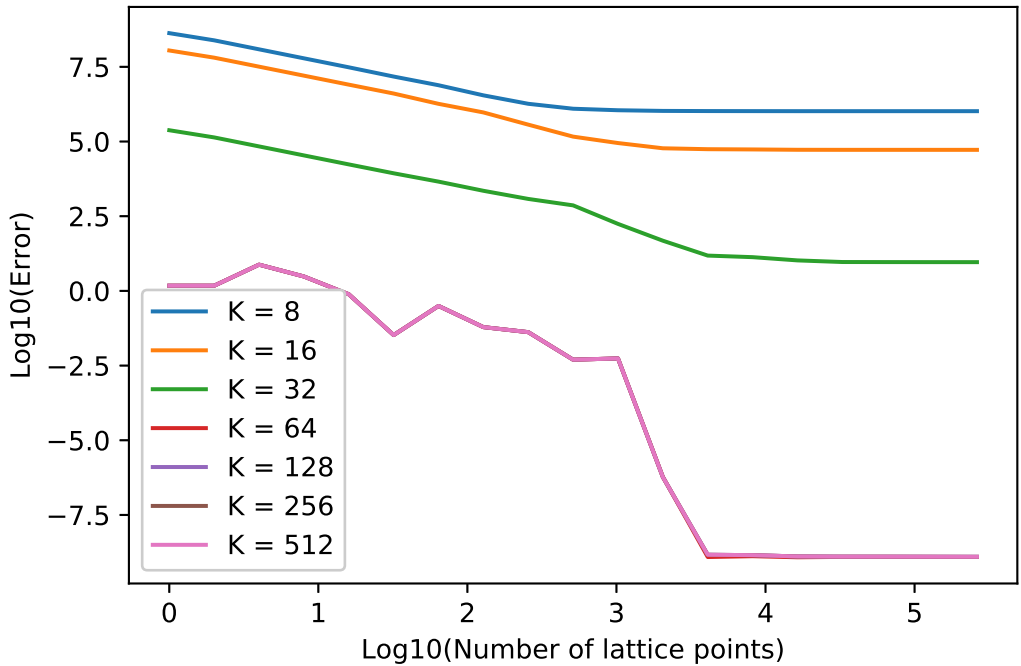}
	\caption{Absolute error verses the number of lattice points for the normal distribution with varying $K$. }\label{figure_normal_kernel_bound}
\end{figure}

Next, we can see the result of changing the kernel truncation factor $K$ in Figure \ref{figure_normal_kernel_bound}.
The projection domain is fixed as $[-4.5, 4.5]^2$.
Once again, the approximation converges to the final limit when the number of lattice points reaches $2^{12}$.
This again shows that the kernel truncation is only related to the final limit of approximation but not to the convergence behavior.

Another point of interest for us is the exponential decay of the Fourier transforms (the building blocks of the reproducing kernel).
When $K \geq 64$, including more terms in the reproducing kernel does not improve the approximation. 
This was one of the main advantages of the COS method, making use of the smoothness of the probability measure to achieve better convergence results. 
It seems that we can arrive at similar results for this combined scheme.

However, the construction for the truncated kernel, defined in Equation \eqref{equation_truncated_kernel}, still suffers from the curse of dimensionality.
Under our current setting, this issue can be addressed through software engineering by making use of an online/offline construction, in which we calculate and store the reproducing kernel in an offline stage and then use the stored results in actual calculation. 
Parallel computing or GPU computing are also possible means to speed up the computation.

We believe that this scheme is valuable for specifically making use of the smoothness of the probability measure and the clear separation of the measure from the integrand, which implies we may switch one part without affecting the other.
Possible remedies to be studied for this scheme may include hyperbolic cross or non-linear basis constructions for the truncated kernel, or alternatively, different means for approximating the reproducing kernel. 

To conclude, our additional parameter of kernel truncation and domain projection does not change the convergence behavior but only affects the final approximation error when it is sufficiently large. 
However, the kernel construction still gives rise to the curse of dimensionality.

\subsection{Asymmetric Multivariate Laplace Distribution}

Finally, we implement our scheme to a typical asymmetric 2-dimensional Laplace distribution ${\bf Y}$, whose character function is given by $$\mathcal{F}_{Lap}({\bf k}) = \frac{1}{1 + 0.5 {\bf k}^\top \boldsymbol{\Sigma}{\bf k} - \imath \bar{\boldsymbol{\mu}}^\top {\bf k}},$$ in this subsection.
In our setting, $\bar{\boldsymbol{\mu}}$ is a 2-dimensional vector given by $(0.3, -0.1)^\top$, while $\boldsymbol{\Sigma}$ is a $2 \time 2 $ matrix defined as 
$$\left(\begin{array}{cc}
0.25 & -0.15 \\ -0.15 & 0.75
\end{array}\right).$$
The mean and covariance of this distribution are given by $\bar{\boldsymbol{\mu}}$ and $\boldsymbol{\Sigma} + \bar{\boldsymbol{\mu}}\bar{\boldsymbol{\mu}}^\top$, respectively.
For further information, readers are referred to \cite{kozubowski_2013}.
In this test, we compare the algorithm in Section \ref{section_scheme} and the extended COS wavelet scheme in Section \ref{section_cosine_wavelets}.

In this example, the integrand is simply $Y_1Y_2$ and the analytic solution for the expectation is $\bar{\mu}_1\bar{\mu_2} + (\boldsymbol{\Sigma} + \bar{\boldsymbol{\mu}}\bar{\boldsymbol{\mu}}^\top)_{1, 2} = -0.21$. 

We pick the projection domain $[\bar{\mu}_1 - 20 \Sigma_{1, 1}, \bar{\mu}_1 + 20 \Sigma_{1, 1}] \times [\bar{\mu}_2 - 20 \Sigma_{2, 2}, \bar{\mu}_2 + 20 \Sigma_{2, 2}]$ and set the kernel truncation parameter $K$ as $2^6$ for the cosine expansion lattice scheme. 
For the COS wavelet scheme, the order of the kernel and the number of lattice points are linked. 
Therefore, when we use $N$ lattice points for the lattice scheme, we pick $N'= \left\lceil\sqrt{N}\right\rceil$ for the respective wavelet test. 
We report in Figure \ref{figure_laplace_distribution}. 

\begin{figure}[t!]
	\includegraphics{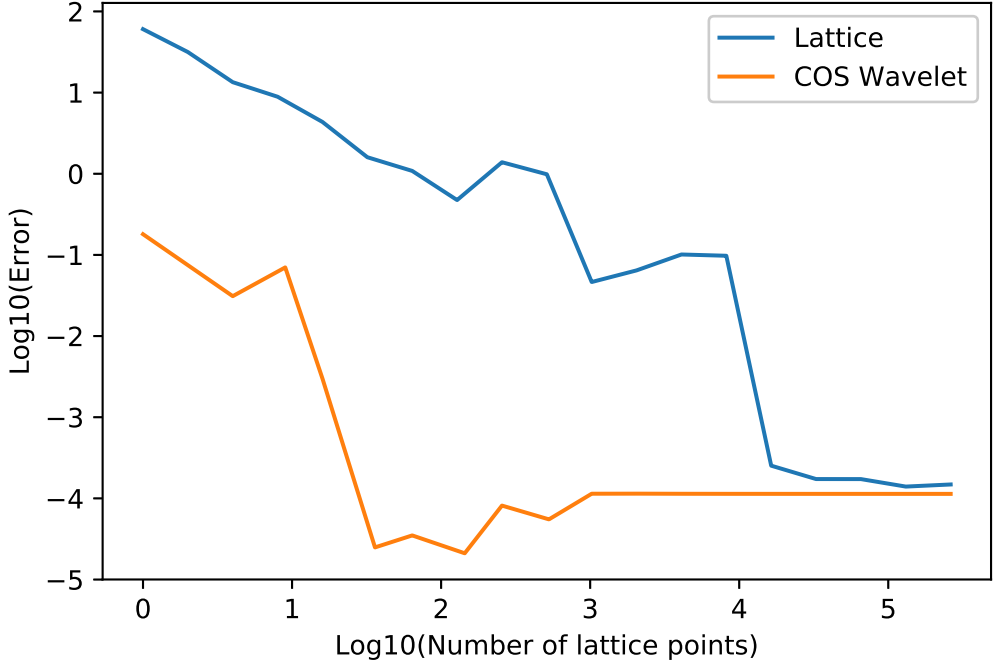}
	\caption{Absolute error verses the number of lattice points for Laplace distribution.}\label{figure_laplace_distribution}
\end{figure}

It is clear that both schemes converge to the same limit.
The approximation error is dominated by the measure truncation error when the number of lattice points is high enough and therefore they have similar end results.
Comparing both schemes, it seems that the COS wavelet performs better with a similar number of lattice points at the early stage.
This may suggest that its quadrature point distribution is more efficient.
The difference in quadrature points for the two schemes can be seen in Figure \ref{figure_point_distribution}.

\begin{figure}[t!]
	\includegraphics{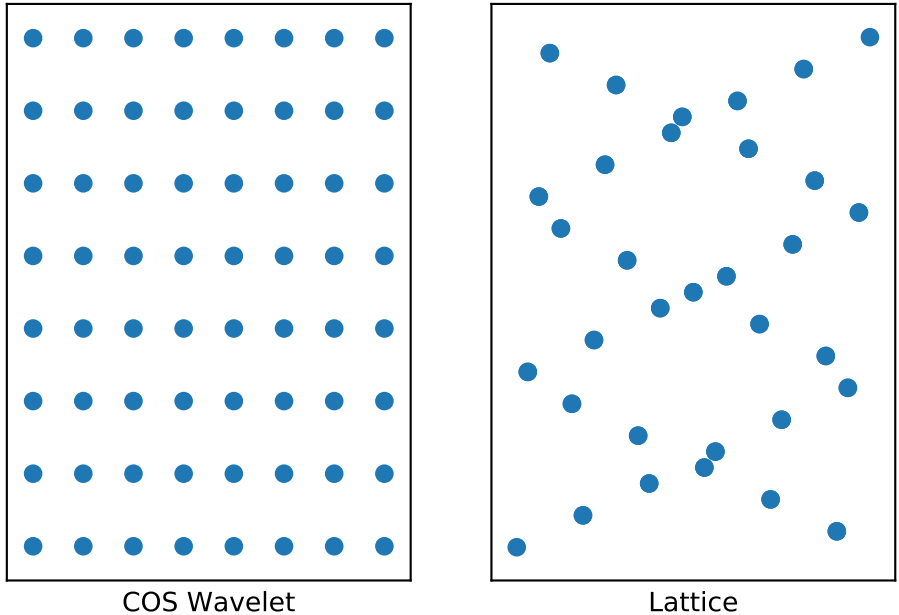}
	\caption{The quadrature points distribution for COS wavelet(left) and lattice sequence(right).}\label{figure_point_distribution}
\end{figure}

However, there are still some shortcomings in the construction of the COS wavelet that makes it a less attractive option comparing to the lattice scheme.
For example, the COS wavelet is not extendable in the sense that we cannot simply add an extra evaluation point to improve the approximation. 
A new evaluation from the beginning is required in such case.

Nevertheless, this test suggests that it would be of great interest to implement other quadrature rules with the half-cosine reproducing kernel.

Another point of interest is that in fact the cosine transform decay rate $\beta =2$ for this example, so the rough error bound for the last term in Equation \ref{equation_total_error_bound} fails to converge. 
However, the success in applying our algorithm clearly shows that we underestimate the decay rate of the cosine transform and a tighter error bound may be derived from advanced results in Fourier analysis.
This motivates us to derive better error bound for our scheme.

\section{Conclusion} \label{section_conclusion}

Equipped with the Fourier-cosine projection technique, we extended the lattice sequence from \cite{hickernell_2012_weighted} to the broader space of probability measures (instead of just the uniform distribution in \cite{dick_2014_lattice}).
The resulting approximation is reproducible and, in theory, this scheme can be generalized to higher-dimensional space by the extension of the results from the previous lattice literature.

The fact that our scheme concentrates all the information from the probability measure into the reproducing kernel and simply evaluates the integrand at some preset quadrature points gives us flexibility. 
We may use the same set of quadrature points for different probability measures as our derivation is not measure specific. Moreover, our scheme can be adjusted and possibly improved by simply replacing the lattice sequence by another one. 

The main remaining issue is the approximation of the reproducing kernel which still suffers from the curse of dimensionality. However, with the rapid decay of the Fourier transform and the specific form of the reproducing kernel, the authors believe that further research on this topic would be fruitful.

\bibliographystyle{abbrv}
\bibliography{../library.bib} 

\begin{thebibliography}{10}

\bibitem{Aronszain_1950_theory}
N.~Aronszajn.
\newblock Theory of reproducing kernels.
\newblock {\em Transactions of the American Mathematical Society},
  68(3):337--404, 1950.

\bibitem{borovykh_2018_efficient}
A.~Borovykh, A.~Pascucci, and C.~W. Oosterlee.
\newblock Efficient computation of various valuation adjustments under local
  {L}\'evy models.
\newblock {\em SIAM Journal on Financial Mathematics}, 9(1):251--273, 2018.

\bibitem{chau_2018_wavelet}
K.~W. Chau and C.~W. Oosterlee.
\newblock {On the wavelet-based SWIFT method for backward stochastic
  differential equations}.
\newblock {\em IMA Journal of Numerical Analysis}, 38(2):1051--1083, 04 2018.

\bibitem{chau_2015_ruin}
K.~W. Chau, S.~C.~P. Yam, and H.~Yang.
\newblock Fourier-cosine method for ruin probabilities.
\newblock {\em Journal of Computational and Applied Mathematics}, 281:94 --
  106, 2015.

\bibitem{cools_2006_constructing}
R.~Cools, F.~Y. Kuo, and D.~Nuyens.
\newblock Constructing embedded lattice rules for multivariate integration.
\newblock {\em SIAM Journal on Scientific Computing}, 28(6):2162--2188, 2006.

\bibitem{dick_2004_on}
J.~Dick.
\newblock On the convergence rate of the component-by-component construction of
  good lattice rules.
\newblock {\em Journal of Complexity}, 20(4):493 -- 522, 2004.

\bibitem{dick_2014_lattice}
J.~Dick, D.~Nuyens, and F.~Pillichshammer.
\newblock Lattice rules for nonperiodic smooth integrands.
\newblock {\em Numerische Mathematik}, 126(2):259--291, Feb 2014.

\bibitem{dick_2010_digital}
J.~Dick and F.~Pillichshammer.
\newblock {\em Digital Nets and Sequences: Discrepancy Theory and Quasi–Monte
  Carlo Integration}.
\newblock Cambridge University Press, 2010.

\bibitem{dick_2008_construction}
J.~Dick, F.~Pillichshammer, and B.~J. Waterhouse.
\newblock The construction of good extensible rank-1 lattices.
\newblock {\em Mathematics of Computation}, 77(264):2345--2373, 2008.

\bibitem{fang_2009_novel}
F.~Fang and C.~W. Oosterlee.
\newblock A novel pricing method for {E}uropean options based on
  {F}ourier-cosine series expansions.
\newblock {\em SIAM Journal on Scientific Computing}, 31(2):826--848, 2009.

\bibitem{fang_2009_pricing}
F.~Fang and C.~W. Oosterlee.
\newblock Pricing early-exercise and discrete barrier options by
  {F}ourier-cosine series expansions.
\newblock {\em Numerische Mathematik}, 114(1):27--62, Aug 2009.

\bibitem{fischer_1997_wavelets}
B.~Fischer and J.~Prestin.
\newblock Wavelets based on orthogonal polynomials.
\newblock {\em Mathematics of Computation}, 66(220):1593--1618, 1997.

\bibitem{hickernell_2012_weighted}
F.~J. Hickernell, P.~Kritzer, F.~Y. Kuo, and D.~Nuyens.
\newblock Weighted compound integration rules with higher order convergence for
  all {N}.
\newblock {\em Numerical Algorithms}, 59(2):161--183, Feb 2012.

\bibitem{hickernell_2003_existence}
F.~J. Hickernell and H.~Niederreiter.
\newblock The existence of good extensible rank-1 lattices.
\newblock {\em Journal of Complexity}, 19(3):286 -- 300, 2003.
\newblock Oberwolfach Special Issue.

\bibitem{kozubowski_2013}
T.~J. Kozubowski, K.~Podg\'orski, and I.~Rychlik.
\newblock Multivariate generalized {L}aplace distribution and related random
  fields.
\newblock {\em Journal of Multivariate Analysis}, 113:59 -- 72, 2013.
\newblock Special Issue on Multivariate Distribution Theory in Memory of Samuel
  Kotz.

\bibitem{kuo_2002_component}
F.~Y. Kuo and S.~Joe.
\newblock Component-by-component construction of good lattice rules with a
  composite number of points.
\newblock {\em Journal of Complexity}, 18(4):943 -- 976, 2002.

\bibitem{lisi_2007_some}
M.~Lisi.
\newblock Some remarks on the {C}antor pairing function.
\newblock {\em Le Matematiche}, 62(1):55--65, Dec 2007.

\bibitem{niedereriter_1992_random}
H.~Niederreiter.
\newblock {\em Random Number Generation and Quasi-Monte Carlo Methods}.
\newblock Society for Industrial and Applied Mathematics, 1992.

\bibitem{nuyens_2006_fast}
D.~Nuyens and R.~Cools.
\newblock Fast algorithms for component-by-component construction of rank-1
  lattice rules in shift-invariant reproducing kernel {H}ilbert spaces.
\newblock {\em Mathematics of Computation}, 75(254):903--920, 2006.

\bibitem{nuyens_2006_fast_non_prime}
D.~Nuyens and R.~Cools.
\newblock Fast component-by-component construction of rank-1 lattice rules with
  a non-prime number of points.
\newblock {\em Journal of Complexity}, 22(1):4 -- 28, 2006.
\newblock Special Issue.

\bibitem{ruijter_2012_two}
M.~Ruijter and C.~W. Oosterlee.
\newblock Two-dimensional fourier cosine series expansion method for pricing
  financial options.
\newblock {\em SIAM Journal on Scientific Computing}, 34(5):B642--B671, 2012.

\bibitem{ruijter_2015_fourier}
M.~Ruijter and C.~W. Oosterlee.
\newblock A {F}ourier cosine method for an efficient computation of solutions
  to {BSDEs}.
\newblock {\em SIAM Journal on Scientific Computing}, 37(2):A859--A889, 2015.

\bibitem{sloan_2002_on}
I.~H. Sloan, F.~Y. Kuo, and S.~Joe.
\newblock On the step-by-step construction of quasi-monte carlo integration
  rules that achieve strong tractability error bounds in weighted {S}obolev
  spaces.
\newblock {\em Mathematics of Computation}, 71(240):1609--1640, 2002.

\bibitem{sloan_2002_component}
I.~H. Sloan and A.~V. Reztsov.
\newblock Component-by-component construction of good lattice rules.
\newblock {\em Math. Comput.}, 71(237):263--273, Jan. 2002.

\end{thebibliography}

\appendix

\section{Detail results for some numerical tests}

\subsection{Results for projection domain test}
\begin{tabular}{|*{7}{c|}}
	\hline
	\diagbox{L}{Lattice} & $2^{0}$ & $2^{1}$ & $2^{2}$ & $2^{3}$ & $2^{4}$ & $2^{5}$\\
	\hline
	1 & 3.420E-01 & 4.493E-01 & 7.891E-01 & 8.948E-01 & 9.336E-01 & 9.432E-01\\
	3 & 5.916E-01 & 4.472E-01 & 3.695E-02 & 1.034E-01 & 4.029E-02 & 5.155E-02\\
	5 & 1.490E+00 & 1.490E+00 & 1.303E+00 & 1.424E-01 & 4.891E-01 & 2.301E-01\\
	7 & 1.490E+00 & 1.490E+00 & 3.985E+00 & 1.257E+00 & 1.140E-01 & 1.487E-01\\
	9 & 1.490E+00 & 1.490E+00 & 7.560E+00 & 3.035E+00 & 7.726E-01 & 3.318E-02\\
	11 & 1.490E+00 & 1.490E+00 & 1.203E+01 & 5.270E+00 & 1.890E+00 & 2.730E-01\\
	13 & 1.491E+00 & 1.490E+00 & 1.739E+01 & 7.951E+00 & 3.231E+00 & 8.776E-01\\
	\hline
	\diagbox{L}{Lattice} & $2^{6}$ & $2^{7}$ & $2^{8}$ & $2^{9}$ & $2^{10}$ & $2^{11}$\\
	\hline
	1 & 9.459E-01 & 9.481E-01 & 9.484E-01 & 9.486E-01 & 9.486E-01 & 9.486E-01\\
	3 & 9.309E-02 & 8.966E-02 & 9.249E-02 & 9.144E-02 & 9.206E-02 & 9.201E-02\\
	5 & 1.087E-01 & 1.191E-02 & 3.909E-03 & 1.526E-04 & 1.549E-04 & 1.366E-04\\
	7 & 2.085E-02 & 1.559E-02 & 8.520E-03 & 1.754E-04 & 1.919E-04 & 2.317E-09\\
	9 & 3.135E-01 & 6.133E-02 & 4.157E-02 & 4.981E-03 & 5.486E-03 & 5.855E-07\\
	11 & 7.037E-01 & 1.863E-01 & 1.013E-01 & 2.747E-02 & 3.034E-02 & 6.746E-05\\
	13 & 1.074E+00 & 1.597E-01 & 7.080E-02 & 1.443E-02 & 2.276E-02 & 3.608E-04\\
	\hline
	\diagbox{L}{Lattice} & $2^{12}$ & $2^{13}$ & $2^{14}$ & $2^{15}$ & $2^{16}$ & $2^{17}$\\
	\hline
	1 & 9.486E-01 & 9.486E-01 & 9.486E-01 & 9.486E-01 & 9.486E-01 & 9.486E-01\\
	3 & 9.202E-02 & 9.200E-02 & 9.201E-02 & 9.201E-02 & 9.201E-02 & 9.201E-02\\
	5 & 1.366E-04 & 1.363E-04 & 1.363E-04 & 1.363E-04 & 1.363E-04 & 1.363E-04\\
	7 & 2.348E-09 & 2.357E-09 & 2.359E-09 & 2.359E-09 & 2.358E-09 & 2.359E-09\\
	9 & 1.470E-09 & 1.429E-09 & 1.293E-09 & 1.271E-09 & 1.268E-09 & 1.263E-09\\
	11 & 2.742E-08 & 2.813E-08 & 2.778E-08 & 2.797E-08 & 2.788E-08 & 2.785E-08\\
	13 & 1.182E-07 & 1.128E-07 & 9.505E-08 & 9.574E-08 & 9.694E-08 & 9.578E-08\\

	\hline
\end{tabular}

\subsection{Results for kernel truncation test}
\begin{tabular}{|*{7}{c|}}
	\hline
	\diagbox{K}{Lattice} & $2^{0}$ & $2^{1}$ & $2^{2}$ & $2^{3}$ & $2^{4}$ & $2^{5}$\\
	\hline
	8 & 4.243E+08 & 2.439E+08 & 1.219E+08 & 6.097E+07 & 3.048E+07 & 1.496E+07\\
	16 & 1.115E+08 & 6.410E+07 & 3.205E+07 & 1.603E+07 & 8.013E+06 & 4.044E+06\\
	32 & 2.387E+05 & 1.372E+05 & 6.860E+04 & 3.430E+04 & 1.715E+04 & 8.643E+03\\
	64 & 1.490E+00 & 1.490E+00 & 7.560E+00 & 3.035E+00 & 7.726E-01 & 3.318E-02\\
	128 & 1.490E+00 & 1.490E+00 & 7.560E+00 & 3.035E+00 & 7.726E-01 & 3.318E-02\\
	256 & 1.490E+00 & 1.490E+00 & 7.560E+00 & 3.035E+00 & 7.726E-01 & 3.318E-02\\
	512 & 1.490E+00 & 1.490E+00 & 7.560E+00 & 3.035E+00 & 7.726E-01 & 3.318E-02\\
	\hline
	\diagbox{K}{Lattice} & $2^{6}$ & $2^{7}$ & $2^{8}$ & $2^{9}$ & $2^{10}$ & $2^{11}$\\
	\hline
	8 & 7.660E+06 & 3.513E+06 & 1.832E+06 & 1.252E+06 & 1.114E+06 & 1.059E+06\\
	16 & 1.826E+06 & 9.412E+05 & 3.674E+05 & 1.458E+05 & 8.894E+04 & 5.939E+04\\
	32 & 4.547E+03 & 2.243E+03 & 1.198E+03 & 7.293E+02 & 1.749E+02 & 4.789E+01\\
	64 & 3.135E-01 & 6.133E-02 & 4.157E-02 & 4.981E-03 & 5.486E-03 & 5.939E-07\\
	128 & 3.135E-01 & 6.133E-02 & 4.157E-02 & 4.981E-03 & 5.486E-03 & 5.855E-07\\
	256 & 3.135E-01 & 6.133E-02 & 4.157E-02 & 4.981E-03 & 5.486E-03 & 5.855E-07\\
	512 & 3.135E-01 & 6.133E-02 & 4.157E-02 & 4.981E-03 & 5.486E-03 & 5.855E-07\\
	\hline
	\diagbox{K}{Lattice} & $2^{12}$ & $2^{13}$ & $2^{14}$ & $2^{15}$ & $2^{16}$ & $2^{17}$\\
	\hline
	8 & 1.046E+06 & 1.041E+06 & 1.038E+06 & 1.038E+06 & 1.037E+06 & 1.037E+06\\
	16 & 5.536E+04 & 5.430E+04 & 5.287E+04 & 5.269E+04 & 5.266E+04 & 5.265E+04\\
	32 & 1.520E+01 & 1.346E+01 & 1.049E+01 & 9.260E+00 & 9.210E+00 & 9.180E+00\\
	64 & 1.235E-09 & 1.316E-09 & 1.211E-09 & 1.246E-09 & 1.247E-09 & 1.243E-09\\
	128 & 1.470E-09 & 1.429E-09 & 1.293E-09 & 1.271E-09 & 1.268E-09 & 1.263E-09\\
	256 & 1.470E-09 & 1.429E-09 & 1.293E-09 & 1.271E-09 & 1.268E-09 & 1.263E-09\\
	512 & 1.470E-09 & 1.429E-09 & 1.293E-09 & 1.271E-09 & 1.268E-09 & 1.263E-09\\
	\hline
\end{tabular}

\end{document}